\documentclass[11pt,reqno]{amsart}

\newtheorem{proposition}{Proposition}[section]
\newtheorem{lemma}[proposition]{Lemma}
\newtheorem{corollary}[proposition]{Corollary}
\newtheorem{theorem}[proposition]{Theorem}

\theoremstyle{definition}
\newtheorem{definition}[proposition]{Definition}
\newtheorem{example}[proposition]{Example}
\newtheorem{examples}[proposition]{Examples}
\newtheorem{remark}[proposition]{Remark}

\newcommand{\thlabel}[1]{\label{th:#1}}
\newcommand{\thref}[1]{Theorem~\ref{th:#1}}
\newcommand{\selabel}[1]{\label{se:#1}}
\newcommand{\seref}[1]{Section~\ref{se:#1}}
\newcommand{\lelabel}[1]{\label{le:#1}}
\newcommand{\leref}[1]{Lemma~\ref{le:#1}}
\newcommand{\prlabel}[1]{\label{pr:#1}}
\newcommand{\prref}[1]{Proposition~\ref{pr:#1}}
\newcommand{\colabel}[1]{\label{co:#1}}
\newcommand{\coref}[1]{Corollary~\ref{co:#1}}
\newcommand{\relabel}[1]{\label{re:#1}}

\newcommand{\exlabel}[1]{\label{ex:#1}}
\newcommand{\exref}[1]{Example~\ref{ex:#1}}
\newcommand{\delabel}[1]{\label{de:#1}}
\newcommand{\deref}[1]{Definition~\ref{de:#1}}
\newcommand{\eqlabel}[1]{\label{eq:#1}}
\newcommand{\equref}[1]{(\ref{eq:#1})}

\def\ot{\otimes}

\def\CC{{\mathbb C}}

\newcommand{\Cc}{\mathcal{C}}

\def\*C{{}^*\hspace*{-1pt}{\Cc}}
\def\text#1{{\rm {\rm #1}}}

\input xy
\xyoption {all} \CompileMatrices

\usepackage{amssymb}
\usepackage{color,amssymb,graphicx,amscd,amsmath}
\usepackage[colorlinks,urlcolor=blue,linkcolor=blue,citecolor=blue]{hyperref}

\begin{document}
\title[On a type of commutative algebras]
{On a type of commutative algebras}

\author{A. L. Agore}
\address{Faculty of Engineering, Vrije Universiteit Brussel, Pleinlaan 2, B-1050 Brussels, Belgium \textbf{and} Department of Applied Mathematics, Bucharest
University of Economic Studies, Piata Romana 6, RO-010374
Bucharest 1, Romania} \email{ana.agore@vub.ac.be and
ana.agore@gmail.com}

\author{G. Militaru}
\address{Faculty of Mathematics and Computer Science, University of Bucharest, Str.
Academiei 14, RO-010014 Bucharest 1, Romania}
\email{gigel.militaru@fmi.unibuc.ro and gigel.militaru@gmail.com}

\thanks{A.L. Agore is Postoctoral Fellow of the Fund for Scientific
Research-Flanders (Belgium) (F.W.O. Vlaanderen). This work was
supported by a grant of the Romanian National Authority for
Scientific Research, CNCS-UEFISCDI, grant no. 88/05.10.2011.}


\maketitle

\begin{abstract} We introduce some basic concepts
for Jacobi-Jordan algebras such as: representations, crossed
products or Frobenius/metabelian/co-flag objects. A new family of
solutions for the quantum Yang-Baxter equation is constructed
arising from any $3$-step nilpotent Jacobi-Jordan algebra. Crossed
products are used to construct the classifying object for the
extension problem in its global form. For a given Jacobi-Jordan
algebra $A$ and a given vector space $V$ of dimension
$\mathfrak{c}$, a global non-abelian cohomological object
${\mathbb G} {\mathbb H}^{2} \, (A, \, V)$ is constructed: it
classifies, from the view point of the extension problem, all
Jacobi-Jordan algebras that have a surjective algebra map on $A$
with kernel of dimension $\mathfrak{c}$. The object ${\mathbb G}
{\mathbb H}^{2} \, (A, \, k)$ responsible for the classification
of co-flag algebras is computed, all $1 + {\rm dim} (A)$
dimensional Jacobi-Jordan algebras that have an algebra surjective
map on $A$ are classified and the automorphism groups of these
algebras is determined. Several examples involving special sets of
matrices and symmetric bilinear forms as well as equivalence
relations between them (generalizing the isometry relation) are
provided.
\end{abstract}

\section*{Introduction}
Jacobi-Jordan algebras (JJ algebras for short) were recently
introduced in \cite{BF} as vector spaces $A$ over a field $k$,
equipped with a bilinear map $\cdot : A \times A \to A$ satisfying
the Jacobi identity and instead of the skew-symmetry condition
valid for Lie algebras we impose commutativity $x \cdot y = y
\cdot x$, for all $x$, $y\in A$. These type of algebras appeared
already in relation with Bernstein algebras in 1987 (\cite{Buse}).
One crucial remark is that JJ algebras are examples of the more
popular and well-referenced Jordan algebras \cite{jordan2}
introduced in order to achieve an axiomatization for the algebra
of observables in quantum mechanics. In \cite{BF} the authors
achieved the classification of these algebras up to dimension $6$
over an algebraically closed field of characteristic different
from $2$ and $3$. Our purpose is to introduce and develop some
basic concepts for JJ algebras which might eventually lead to an
interesting theory. As it was explained in  \cite{BF} and as it
will be obvious from this paper as well, JJ algebras are objects
fundamentally different from both associative and Lie algebras
even if their definition differs from the latter only modulo a
sign. We aim to prove that there exists a rich and very
interesting theory behind the Jacobi-Jordan algebras which
deserves to be developed further mainly for three reasons. The
first reason is a theoretical one: JJ algebras are objects of
study in their own right as objects living at the interplay
between the intensively studied Lie algebras and respectively
associative algebras. In this context it is a challenge to
introduce the JJ algebra counterparts of concepts already defined
in the theory of Lie (resp. associative) algebras and to see which
of the results valid in the fields mentioned above are also true
for JJ algebras. The second reason comes from the observation that
JJ algebras are a special class of Jordan algebras which turned
out to play a fundamental role not only in quantum mechanics but
also in differential geometry, algebraic geometry or functional
analysis \cite{Mc} -- from this perspective they deserve a
detailed study. Finally, the third reason is given by the
connection which we will highlight in \seref{basic} between JJ
algebras and the celebrated quantum Yang-Baxter equation from
theoretical physics. A central open problem in this context is to
construct new families of solutions: here we take a first step
towards it by associating to any $3$-step nilpotent JJ algebra a
new family of solutions for the quantum Yang-Baxter equation.
Therefore, constructing such algebras becomes a matter of
interest.

The paper is organized as follows. The first section fixes
notations and conventions used throughout and introduces some
basic concepts in the context of JJ algebras such as modules,
representations or Frobenius objects. In order to find the right
axiom for defining modules over a JJ algebra (\deref{moduleJJ}) we
use the classical trick used for objects $A$ in a $k$-linear
category ${\mathcal C}$: that is, a vector space $V$ with a
bilinear map $\triangleright : A \times V \to V$ such that
$A\times V$ with the 'semi-direct' product type multiplication has
to be an object \emph{inside} the category ${\mathcal C}$.
Representations of a JJ-algebra $A$ (\deref{refJJ}) are introduced
exactly as the representations of a Jordan algebra, if we see
JJ-algebras as a special case of them -- this definition fits with
the way modules were defined, i.e. there exists an isomorphism of
categories between modules and representations over a JJ algebra.
Having introduced these concepts, Frobenius JJ algebras
(\deref{deffrobjj}) arise naturally as symmetric objects in the
category: that is, finite dimensional JJ algebras $A$ such that $A
\cong A^*$, in the category of modules over $A$. Thus, the
definition is similar to the one adopted in the case of the
intensively studied associative Frobenius algebras \cite{frob,
kadison, Kock} and Lie Frobenius algebras (also known in the
literature as self-dual, metric or Lie algebras with an invariant
non-degenerate bilinear form \cite{kat, medina, fig, pelc}). In
\seref{crossed} we deal with the main question addressed in this
paper, namely the \emph{global extension (GE) problem}, and the
crossed product for JJ algebras is introduced as the object
responsible for it. Introduced for Leibniz algebras in
\cite{Mi2013} and studied for associative/Poisson algebras in
\cite{am-2015b, am-2015}, the GE problem is a generalization of
the classical H\"{o}lder's extension problem, and for JJ algebras
consists of the following question:

\emph{Let $A$ be a JJ algebra, $E$ a vector space and $\pi : E \to
A$ a linear epimorphism of vector spaces. Describe and classify
the set of all JJ algebra structures that can be defined on $E$
such that $\pi : E \to A$ becomes a morphism of JJ algebras.}

\prref{hocechiv} proves that any such JJ algebra structure
$\cdot_E$ on $E$ is isomorphic to a \emph{crossed product} $A \# V
= A \#_{(\triangleright, \, \vartheta, \, \cdot_V)} V$, which is a
JJ algebra associated to $A$ and $V := {\rm Ker} (\pi)$ connected
by an action $\triangleright : A \times V \to V$, a symmetric
cocycle $\vartheta: A \times A \to V$ and a JJ algebra structure
$\cdot_V$ on $V$ satisfying some natural axioms as stated in
\prref{hocprod}. The main result of the section, that gives the
theoretical answer to the GE-problem, is proven in
\thref{main1222}: the classifying object for the GE problem (i.e.
the set parameterizes all JJ algebra structures on $E$ up to an
isomorphism which stabilizes $V$ and co-stabilizes $A$) is
parameterized by an explicitly constructed global non-abelian
cohomological type object denoted by ${\mathbb G} {\mathbb H}^{2}
\, (A, \, V)$. \coref{desccompcon} proves that ${\mathbb G}
{\mathbb H}^{2} \, (A, \, V)$ is the coproduct of all non-abelian
cohomologies ${\mathbb H}^{2} \, \, (A, \, (V, \cdot_V))$, the
latter being the Jacobi-Jordan classifying object for the
classical extension problem. In the 'abelian' case (corresponding
to the trivial multiplication $\cdot_V : = 0$ on $V$)
representations of $A$ are used in order to give the decomposition
of ${\mathbb H}^{2} \, \, (A, \, (V, \cdot_V := 0))$ as a
coproduct over all $A$-module structures on $V$
(\coref{cazuabspargere}): this is the JJ algebra counterpart of
the Hochschild theorem for associative algebras \cite[Theorem
6.2]{Hoch2} respectively of Chevalley and Eilenberg's results for
Lie algebras \cite[Theorem 26.2]{CE}. Finally, in
\seref{aplicatii} our theoretical results from the previous
section are applied to two important classes of JJ algebras namely
co-flag algebras (\deref{coflaglbz}) and metabelian algebras. One
of our motivations for paying special attention to metabelian
algebras is that the concept turned out to play a central role in
many other fields; for instance, metabelian groups play a key role
in the study of periodic groups or the Brauer groups as well as in
the isomorphism problem. Furthermore, all JJ algebras of small
dimension classified in \cite{BF} are metabelian. As a concrete
example it is shown in \exref{heiscof} that ${\mathbb G} {\mathbb
H}^{2} \, (\mathfrak{h} (2n + 1, k), \, k) \, \cong \, {\rm Sym}
(n, k) \times {\rm Sym} (n, k) \times {\rm Sym} (n, k)/kI_n$,
where $\mathfrak{h} (2n + 1, k)$ is the Heisenberg algebra and
${\rm Sym} (n, k)$ is the space of all $n\times n$ symmetric
matrices. \prref{dim1der} and respectively \prref{codim1der}
describe and classify all metabelian JJ algebras having the
derived algebra of dimension $1$ respectiv of co-dimension $1$.
Moreover, the explicit description of all $(n+1)$-dimensional
metabelian JJ algebras having the derived algebra of dimension $1$
or codimension $1$ is provided.

\section{Jacobi-Jordan algebras: basic concepts} \selabel{basic}
\subsection*{Notations and terminology}
For a family of sets $(X_i)_{i\in I}$ we shall denote by
$\sqcup_{i\in I} \, X_i$ their coproduct in the category of sets,
i.e. $\sqcup_{i\in I} \, X_i$ is the disjoint union of the
$X_i$'s. Unless otherwise specified, all algebraic entities
(vector spaces, linear or bilinear maps etc.) are over an
arbitrary field $k$. A linear map $f: W \to V$ between two vector
spaces is called \emph{trivial} if $f (x) = 0$, for all $x\in W$.
A bilinear map $\vartheta : W\times W \to V$ is symmetric if
$\vartheta (x, \, y) = \vartheta (y, \, x)$, for all $x$, $y \in
W$. Throughout this paper, by an \emph{algebra} we mean a pair $A
= (A, \cdot)$ consisting of a vector space and a bilinear map
$\cdot : A\times A \to A$ called the multiplication of $A$. If in
addition the multiplication $\cdot$ on $A$ satisfies some
extra-axioms like commutativity, associativity, skew-symmetry and
the Jacobi identity, Leibniz law etc. then $A$ will be called
commutative, associative, Lie or respectively Leibniz algebra,
etc. An algebra $A$ is called \emph{abelian} if its multiplication
is the trivial map, i.e. $a \cdot b = 0$, for all $a$, $b\in A$.
We shall denote by $\sum_{(c)}$ the circular sum - for example, if
$A$ is an algebra and $\vartheta: A \times A \to V$ is a bilinear
map then $\sum_{(c)} \, \vartheta (a, \, b \cdot c) = \vartheta
(a, \, b \cdot c) + \vartheta (b, \, c\cdot a) + \vartheta (c, \,
a\cdot b)$. The concepts of morphisms of algebras, subalgebras,
two-sided ideals, etc. are defined in the obvious way. If $X$ and
$Y$ are two subspace of an algebra $A$ then $X \cdot Y$ stands for
the subspace generated by all $x \cdot y$, for all $x \in X$ and
$y\in Y$. In particular, $A' := A \cdot A$ is a two-sided ideal of
$A$ called the \emph{derived algebra} of $A$. An algebra $A$ is
called \emph{metabelian} if $A'$ is an abelian subalgebra of $A$,
i.e. $(a\cdot b) \cdot (c \cdot d) = 0$, for all $a$, $b$, $c$, $d
\in A$. The derived series of an algebra $A$ is defined
inductively by $A^{(1)} := A'$ and $A^{(n+1)} := (A^{(n)})'$, for
all $n \geq 1$. An algebra $A$ is called \emph{solvable} of step
$m \geq 1$ if $A^{(m)} = 0$ and $A^{(i)} \neq 0$, for all $i < m$.
Thus, a non-abelian algebra $A$ is metabelian if it is a $2$-step
solvable algebra. The lower central series is the series with
terms given by: $A^1 := A$ and $A^{n+1} := \sum_{i=1}^{n} A^i
\cdot A^{n+1-i}$, for all $n\geq 1$. An algebra $A$ is called
\emph{nilpotent} of step $m>1$ if $A^{m} = 0$ and $A^{i} \neq 0$,
for all $i < m$. A \emph{Leibniz algebra} \cite{LoP} is an algebra
$A = (A, \cdot)$ such that the multiplication $\cdot$ satisfies
the Leibniz law for any $a$, $b$, $c\in A$:
\begin{equation} \eqlabel{lbz}
(a \cdot b) \cdot c = a \cdot (b \cdot c) + (a \cdot c) \cdot b
\end{equation}
Leibniz algebras have become popular and intensively studied in
their own right as non-commutative generalizations of Lie
algebras, but also in connection to homological algebra, classical
or non-commutative differential geometry, vertex operator algebras
or integrable systems. A \emph{Jordan algebra} \cite{Al, jordan2}
is an algebra $A = (A, \circ)$ such that the multiplication
$\circ$ satisfies the following axioms for any $a$, $b\in A$:
\begin{equation}\eqlabel{jorddef}
a \circ b = b \circ a, \qquad (a^2 \circ b) \circ a = a^2 \circ (b
\circ a)
\end{equation}
Any associative algebra $A$ can be endowed with the Jordan algebra
structure given by $a \circ b := ab + ba$, for all $a$, $b\in A$,
where the juxtaposition $ab$ denotes the multiplication of the
associative algebra $A$. We denote this Jordan algebra associated
to the associative algebra $A$ by $A_J$. In particular, for any
vector space $V$, the space of all linear endomorphisms ${\rm
End}_k (A)_J$ is a Jordan algebra. Jordan algebras have been
intensively studied from algebraic point of view as well as for
their applications to quantum mechanics, differential geometry,
algebraic geometry or functional analysis - for details see
\cite{Mc} and the references therein.

\subsection*{Jacobi-Jordan algebras} A \emph{Jacobi-Jordan algebra} \cite{BF}
(a JJ-algebra, for short) is an algebra $A = (A, \cdot)$ such that
for any $a$, $b\in A$:
\begin{equation}\eqlabel{jjdef}
a \cdot b = b \cdot a, \qquad \sum_{(c)} a \cdot (b \cdot c) = 0
\end{equation}
that is, $\cdot$ is commutative and satisfies the Jacobi identity
$a \cdot (b \cdot c) + b\cdot (c \cdot a) + c \cdot (a \cdot b) =
0$. An important feature of JJ-algebras is that they are all
nilpotent. By ${\rm Aut}_{\rm JJ-Alg} (A)$ we will denote the
automorphism group of a JJ algebra $A$: i.e. the group with
respect to the composition of all linear automorphisms $f: A \to
A$ satisfying $f(a\cdot b) = f(a) \cdot f(b)$, for all $a$, $b\in
A$. A clue result \cite[Lemma 2.2]{BF} proves that any JJ-algebra
is a Jordan algebra, that is the second identity of
\equref{jorddef} also holds. As is shown in \cite{BF} even if at
first glance these algebras seem quite similar to Lie algebras,
they are in fact very different: the commutativity of the
multiplication $\cdot$ instead of the skew-symmetry satisfied by
Lie algebras gives rise to a radically different concept. Of
course, any vector space $V$ becomes a JJ algebra with the trivial
multiplication: $ x \cdot y = 0$, for any $x$, $y \in V$. Such a
JJ algebra is called abelian and will be denoted by $V_0$.
Examples and the classification of small dimensional JJ algebras
is given in \cite{BF}. Here we give two examples that will be used
later on. Throughout the paper we will use the following
convention: the multiplication of an algebra $A$ will only be
defined on the elements of its basis and we only write down the
non-zero multiplications.

\begin{examples} \exlabel{heisexample}
1.  The commutative Heisenberg JJ algebra \cite{BF} is the $(2n +
1)$-dimensional algebra $\mathfrak{h} (2n + 1, k)$ having $\{e_1,
\cdots, e_n, f_1, \cdots, f_n, z \}$ as a basis over $k$ and
multiplication defined by $e_i \cdot f_i = f_i \cdot e_i := z$,
for all $i = 1, \cdots, n$.

2. An interesting example of an infinite dimensional JJ algebra
can be constructed as follows. Let $V$ be a vector space, $f \in
{\rm End}_k (V)$ with $f^2 = 0$ and $v_0\in {\rm Ker} (f)$. Then
$V_{(f, \, v_0)} := k \times V$ is a JJ algebra via the
multiplication defined for any $p$, $q \in k$ and $x$, $y \in V$
by:
\begin{equation} \eqlabel{metcodim1}
(p, \, x) \cdot (q, \, y) := (0, \, pq \, v_0 + p f(y) + q f(x) \,
)
\end{equation}
Indeed, since $v_0 \in {\rm Ker} (f)$ and $f^2 = 0$ one can easily
see that $(r, \, z) \cdot \bigl( (p, \, x) \cdot (q, \, y) \bigl)
= 0$, for all $p$, $q$, $r \in k$ and $x$, $y$, $z \in V$ and thus
the Jacobi identity is trivially fulfilled. This shows that
$V_{(f, \, v_0)}$ is a $3$-step nilpotent JJ algebra. It is worth
pointing out that $V_{(f, \, v_0)}$ is not a Lie algebra, if ${\rm
char} (k) \neq 2$ and $(f, \, v_0) \neq (0, 0)$.
\end{examples}

The first important difference between JJ algebras and Lie
algebras is highlighted by the next result which provides a new
class of solutions for the famous quantum Yang-Baxter equation. In
what follows $\ot$ denotes $\ot_k$, the tensor product over $k$
and for a linear map $R \in {\rm End}_k (A \ot A)$ we shall denote
by $R^{12} := R \ot {\rm Id}_A$, $R^{23} := {\rm Id}_A \ot R \in
{\rm End}_k (A \ot A \ot A)$. $R$ is called a solution of the
\emph{quantum Yang-Baxter equation} if $R^{12} R^{23} R^{12} =
R^{23} R^{12} R^{23}$ in ${\rm End}_k (A \ot A \ot A)$. For an
algebra $A = (A, \cdot)$ we shall denote by $Z(A)$ the
\emph{Leibniz center} of $A$, i.e. $ Z(A) := \{ z \in A \, | \,
z\cdot A = A \cdot z = 0\}$. Having in mind that any Lie algebra
is also a Leibniz algebra we can prove the following:

\begin{proposition}\prlabel{qybe}
Let $A = (A, \cdot)$ be an algebra, $\alpha$, $\beta \in k^*$ and
$0 \neq z \in Z(A)$. Then, the linear map
$$
R = R_{\alpha, \, \beta, \, z} : A \ot A \to A \ot A, \qquad R (a
\ot b) := \alpha \, b \ot a + \beta \, z \ot (a\cdot b)
$$
is a solution of the quantum Yang-Baxter equation if and only if
$(A, \cdot)$ is a Leibniz algebra. In particular, if $A$ is a JJ
algebra, then $R_{\alpha, \, \beta, \, z}$ is a solution of the
quantum Yang-Baxter equation if and only if ${\rm char} (k) = 2$
or $A$ is at most $3$-step nilpotent algebra.
\end{proposition}

\begin{proof}
Let $a$, $b$, $c\in A$. Taking into account that $z \in Z(A)$, a
straightforward computation proves that:
\begin{eqnarray*}
&& R^{12} R^{23} R^{12} (a \ot b \ot c) = \alpha^3 \, c \ot b \ot
a + \alpha^2 \beta \, z \ot b\cdot c \ot a + \alpha^2 \beta \, z
\ot b \ot a\cdot c \\
&& +  \, \alpha^2 \beta \, c \ot z \ot a\cdot b + \alpha \beta^2
\, z \ot z \ot (a\cdot b) \cdot c, \qquad {\rm and} \\
&& R^{23} R^{12} R^{23} (a \ot b \ot c) = \alpha^3 \, c \ot b \ot
a + \alpha^2 \beta \, c \ot z \ot a\cdot b + \alpha^2 \beta \, z
\ot b \ot a\cdot c \\
&& +  \, \alpha \beta^2 \, z \ot z \ot (a\cdot c) \cdot b +
\alpha^2 \beta \, z \ot (b\cdot c) \ot a + \alpha \beta^2 \, z \ot
z \ot a \cdot (b\cdot c)
\end{eqnarray*}
Thus, $R$ is a solution of the quantum Yang-Baxter equation if and
only if
$$
\alpha \beta^2 \, z \ot z \ot (a\cdot b) \cdot c =  \alpha \beta^2
\, z \ot z \ot (a\cdot c) \cdot b + \alpha \beta^2 \, z \ot z \ot
a \cdot (b\cdot c)
$$
Since the scalars $\alpha$ and $\beta$ are non-zero and $z \neq
0$, the last equation is equivalent to $(a \cdot b) \cdot c = a
\cdot (b \cdot c) + (a \cdot c) \cdot b$, for all  $a$, $b$, $c\in
A$, that is $(A, \cdot)$ is a Leibniz algebra. Assume now that $A$
is a JJ algebra. Then, using the Jacobi identity and the
commutativity of $\cdot$ we obtain that $R$ is a solution of the
quantum Yang-Baxter equation if and only if $ 2 \, c \cdot (a
\cdot b) = 0$, for all $a$, $b$, $c\in A$, and hence the last
statement also follows.
\end{proof}

An example of a JJ algebra which is at most $3$-step nilpotent is
the commutative Heiseberg algebra $\mathfrak{h} (2n + 1, k)$.
Applying \prref{qybe} we construct a family of solutions of the
quantum Yang-Baxter equation arising from any triple of $n\times
n$ symmetric matrices.

\begin{example}
Let $X = (x_{ij})$, $Y = (y_{ij})$, $Z = (z_{ij})$ be three
$n\times n$ symmetric matrices and $A = A_{X, Y, Z}$ the JJ
algebra having the basis $\{e_1, \cdots, e_n, \, f_1, \cdots, f_n,
\, y, \, z \}$ and the multiplication given for any $i$, $j =
\cdots n$ by:
\begin{eqnarray*}
e_i \cdot e_j = e_j \cdot e_i := x_{ij} \, y, \quad f_i \cdot f_j
= f_j \cdot f_i := y_{ij} \, y, \quad e_i \cdot f_j = f_j\cdot e_i
:= \delta_i^j \, z + z_{ij} \, y
\end{eqnarray*}
where $\delta_i^j$ is the Kronecker symbol. The above algebra will
be explicitly constructed in \exref{heiscof}. Then $A$ is at most
$3$-step nilpotent JJ algebra and the Leibniz center $Z(A_{X, Y,
Z})$ is $2$-dimensional having $\{y, z\}$ as a basis. Using,
\prref{qybe} we obtain that for any sclaras $\alpha$, $\beta$,
$\gamma \in k$ the linear map defined for any $a$, $b \in A$ by:
$$
R : A \ot A \to A \ot A, \qquad R (a \ot b) := \alpha \, b \ot a +
(\beta \, y + \gamma z) \ot (a\cdot b)
$$
is a solution of the quantum Yang-Baxter equation in ${\rm End}_k
(A \ot A \ot A)$.
\end{example}

Now, we shall define modules and representations of a JJ algebra.

\begin{definition}\delabel{moduleJJ}
Let $A$ be a JJ-algebra. A (left) \emph{Jacobi-Jordan $A$-module}
(JJ $A$-module, for short) is a vector space $V$ equipped with a
bilinear map $\triangleright : A \times V \to V$, called action,
such that for any $a$, $b\in A$ and $x\in V$:
\begin{eqnarray}
(a\cdot b) \triangleright x = - a \triangleright (b \triangleright
x) - b \triangleright (a \triangleright x) \eqlabel{JJmod}
\end{eqnarray}
We denote by ${}_A{\mathcal J} {\mathcal J}$ the category of all
(left) Jacobi-Jordan $A$-modules having the action preserving
linear maps as morphisms.
\end{definition}

The category ${\mathcal J} {\mathcal J}_A$ of right JJ $A$-modules
is defined analogously and since $A$ is commutative there exists
an isomorphism of categories ${\mathcal J} {\mathcal J}_A \cong
{}_A{\mathcal J} {\mathcal J}$. For this reason we will drop the
adjective 'left' when we speak about JJ modules.

\begin{remark} \relabel{deceasa}
Axiom \equref{JJmod} from the definition of JJ-modules is
surprisingly different from the one corresponding to modules over
an associative algebra, namely $(a\cdot b) \triangleright x = a
\triangleright (b \triangleright x)$. It is instead more close to
the definition of Lie modules but differs from those by the sign
of the first term in the right hand side. We recall that a (left)
Lie $\mathfrak{g}$-module over a Lie algebra $\mathfrak{g} =
(\mathfrak{g}, [-, \, -])$ is a vector space $V$ together with a
bilinear map $ \triangleright : \mathfrak{g} \times V \to V$ such
that $[g, \, h] \triangleright x = g \triangleright (h
\triangleright x) - h \triangleright (g \triangleright x)$, for
all $g$, $h \in \mathfrak{g}$ and $x\in V$. However, we shall see
in \exref{abelian} that axiom \equref{JJmod} is the correct one
from the classical view point of defining modules over a given
mathematical object ${\mathcal O}$ in a $k$-linear category
${\mathcal C}$: that is, a vector space $V$ with a bilinear map
$\triangleright : {\mathcal O} \times V \to V$ such that
${\mathcal O} \times V$ with the 'semi-direct' product type
multiplication $(a, x) \cdot (b, y) := (ab, \, a \triangleright y
+ b\triangleright x)$ has to be an object \emph{inside} the
$k$-linear category ${\mathcal C}$.
\end{remark}

However, representations of a JJ-algebra $A$ will be defined
exactly as the representations of a Jordan algebra, if we see
JJ-algebras as a special case of them. We recall that we have
denoted by ${\rm End}_k (V)_J$, the Jordan algebra of all
$k$-linear endomorphism of a vector space $V$ viewed as a Jordan
algebra via $(f, \, g) \mapsto f \circ g + g \circ f$, for all
$f$, $g \in {\rm End}_k (V)$.

\begin{definition}\delabel{refJJ}
A \emph{representation} of a JJ-algebra $A$ on a vector space $V$
is a morphism of Jordan algebras $\varphi : A \to {\rm End}_k
(V)_J$, i.e. $\varphi (a \cdot b) = \varphi (a) \circ \varphi (b)
+ \varphi (b) \circ \varphi (a)$, for all $a$, $b\in A$, where
$\circ$ is the usual composition of linear endomorphisms of $V$.
\end{definition}

Representations of a JJ-algebra $A$ and Jacobi-Jordan $A$-modules
are two different ways of describing the same structure: more
precisely, there exists an isomorphism of categories
${}_A{\mathcal J} {\mathcal J} \cong {\rm Rep} (A)$, where ${\rm
Rep} (A)$ is the category of representations of $A$ with the
obvious morphisms. The one-to-one correspondence between JJ
$A$-module structures $\triangleright$ on $V$ and representations
$\varphi $ of $A$ on $V$ is given by the two-sided formula
$\varphi (a) (x) :=: - \, a \triangleright x$, for all $a\in A$
and $x\in V$.

We recall the Ado and Harish-Chandra theorem saying that any Lie
algebra over a field of characteristic zero has a faithful
representation. On the other hand, the fundamental result of
Albert \cite{Al} tell us that the Jordan algebra $H_3
(\textbf{O})$ of Hermitian $3\times 3$-matrices with entries in
the octonions can not be imbedded into a Jordan algebra of the
form $A_J$. Thus, is tempting to ask the following question:

\emph{Let $A$ be a Jacobi-Jordan algebra over a field of
characteristic zero. Does $A$ have a faithful representation, i.e.
there exists a representation $\varphi : A \to {\rm End}_k (V)_J$
of $A$ on a vector space $V$ that is an injective map?}

Since ${}_A{\mathcal J} {\mathcal J} \cong {\rm Rep} (A)$, we
prefer to deal from now on only with JJ $A$-modules instead of
representations. First of all, any vector space $V$ is a JJ
$A$-module via the trivial action $a \triangleright x := 0$, for
any $a \in A$ and $x \in A$. On the other hand, using the Jacobi
identity, we can easily see that $A$ and the linear dual $A^* :=
{\rm Hom}_k (A, k)$ are JJ $A$-modules via the canonical actions:
\begin{equation}\eqlabel{canact}
a \triangleright x := a \cdot x, \qquad (a \triangleright a^*) (x)
:= a^* (a \cdot x)
\end{equation}
for all $a$, $x\in A$ and $a^* \in A^*$. Modules over a JJ algebra
will prove themselves to be useful tools for defining Frobenius
algebras, as the symmetric object in the category of JJ algebras.

\begin{definition}\delabel{deffrobjj}
A finite dimensional Jacobi-Jordan algebra $A$ is called
\emph{Frobenius} if there exists an isomorphism $A \cong A^*$ in
${}_A{\mathcal J} {\mathcal J}$.
\end{definition}

\deref{deffrobjj} complies with the classical definition of
Frobenius \cite{frob} on associative algebras: a JJ algebra $A$ is
Frobenius if and only if the regular representation $A \to {\rm
End}_k (A)_J$, $a \mapsto (b \to  a \cdot b)$ and the canonical
representation ${\rm can}: A \to {\rm End}_k (A^*)_J$, ${\rm can}
(a) (a^*) (b) := a^* (a \cdot b)$, for all $a$, $b\in A$ and $a^*
\in A^*$ are equivalent.

\begin{proposition}\prlabel{caracfrj}
A finite dimensional JJ-algebra $A$ is Frobenius if and only if
there exists a bilinear form $B: A \times A \to k$ that is
nondegenerate and invariant, i.e. $B (a \cdot b, \, c) = B (a, \,
b \cdot c)$, for all $a$, $b$, $c \in A$.
\end{proposition}

\begin{proof}
Follows from the one-to-one correspondence between the set of all
$k$-linear isomorphisms $f: A \to A^*$ and the set of all
nondegenerate bilinear forms $B : A \times A \to k$ given by the
two-sided formula $f (a) (b) :=: B(a, b)$, for all $a$, $b\in A$.
Under this bijection, it is easy to see that the JJ $A$-module
maps $f: A \to A^*$ correspond to those bilinear forms $B: A\times
A \to k$ that are invariant.
\end{proof}

\prref{caracfrj} provides an efficient criterion for testing if a
JJ algebra is Frobenius.

\begin{examples} \exlabel{frobex}
1. Any finite dimensional abelian JJ algebra is Frobenius. Indeed,
as the multiplication on $A$ is trivial, then any linear
isomorphism $f : A \to A^*$ is an isomorphism of JJ $A$-modules.

2. Let $A : = A_{1, 2}$ be the $2$-dimensional JJ algebra with
basis $\{e_1, \, e_2\}$ and the multiplication $e_1 \cdot e_1 =
e_2$. Then $A_{1, 2}$ is a Frobenius JJ algebra since the bilinear
form defined by $B(e_1, \, e_2) = B(e_2, \, e_1) = B(e_1, \, e_1)
:= 1$ and $B(e_2, \, e_2) := 0$ is non-degenerate and invariant.

3. Let $A := \mathfrak{h} (2n + 1, k)$ be the commutative
Heisenberg JJ algebra from \exref{heisexample}. Then $\mathfrak{h}
(2n + 1, k)$ is not Frobenius. Indeed, let $B: \mathfrak{h} (2n +
1, k) \times \mathfrak{h} (2n + 1, k) \to k$ be an invariant
bilinear form. Then $B (z, \, c) = B (e_i \cdot f_i, \, c) = B
(e_i, \, f_i \cdot c)$, for all $c \in \mathfrak{h} (2n + 1, k)$
and $i = 1, \cdots, n$. By taking $c := z$ and $c := f_j$ we
obtain $B(z, \, z) = B(z, \, f_j) = 0$, for any $j = 1, \cdots,
n$. On the other hand, by taking $c := e_j$ and choosing $i \neq
j$ yields $B(z, \, e_j) = 0$. Thus, $B(z, - ) = 0$, for any
invariant form $B$ on $\mathfrak{h} (2n + 1, k)$, that is $B$ is
degenerate and hence the JJ algebra $\mathfrak{h} (2n + 1, k)$ is
not Frobenius.
\end{examples}

A more detailed analysis of Frobenius JJ algebras will be
performed somewhere else. We mention that Frobenius associative
algebras are studied not only from algebraic point of view but
also for their implications in topology, algebraic geometry and 2D
topological quantum field theories, category theory, Hochschild
cohomology or graph theory (\cite{kadison, Kock}). On the other
hand, Frobenius Lie algebras are studied in pure mathematics
\cite{kat, medina} as well as in physics \cite{fig, pelc}.

\section{Crossed products and the global extension problem} \selabel{crossed}

In this section we deal with the GE problem, as formulated in the
introduction, in the context of JJ algebras. Let $A$ be a
JJ-algebra, $E$ a vector space, $\pi : E \to A$ a linear
epimorphism of vector spaces with $V: = {\rm Ker} (\pi)$ and
denote by $i: V \to E$ the inclusion map. We say that a linear map
$\varphi: E \to E$ \emph{stabilizes} $V$ (resp.
\emph{co-stabilizes} $A$) if $\varphi \circ i = i$ (resp. $\pi
\circ \varphi = \pi$). Two JJ algebra structures $\cdot $ and
$\cdot'$ on $E$ such that $\pi : E \to A$ is a morphism of JJ
algebras are called \emph{cohomologous} and we denote this by $(E,
\cdot) \approx (E, \cdot')$, if there exists a JJ algebra map
$\varphi: (E, \cdot) \to (E, \cdot')$ which stabilizes $V$ and
co-stabilizes $A$. One can easily see that any such morphism is
bijective and therefore $\approx$ is an equivalence relation on
the set of all JJ-algebra structures on $E$ such that $\pi : E \to
A$ is a JJ algebra map. The set of all equivalence classes via the
equivalence relation $\approx$ will be denoted by ${\rm Gext} \,
(A, \, E)$ and it is the classifying object for the GE problem. In
the sequel we will prove that ${\rm Gext} \, (A, \, E)$ is
parameterized by a global non-abelian cohomological type object
${\mathbb G} {\mathbb H}^{2} \, (A, \, V)$ which will be
explicitly constructed. To begin with, we introduce the crossed
product of JJ algebras:

\begin{definition} \delabel{hocdat}
Let $A = (A, \cdot)$ be a JJ algebra and $V$ a vector space. A
\emph{crossed data} of $A$ by $V$ is a system $\Theta(A, V) =
(\triangleright, \, \vartheta, \, \cdot_V)$ consisting of three
bilinear maps
$$
\triangleright \, \, : A \times V \to V, \quad \vartheta \,\, : A
\times A \to V, \quad \cdot_V \, : V \times V \to V
$$
\end{definition}

For a crossed data $ \Theta(A, V) = (\triangleright, \, \vartheta,
\, \cdot_V)$ we denote by $A \# V = A \#_{(\triangleright, \,
\vartheta, \, \cdot_V)} V $ the vector space $A \times V$ with the
multiplication $\bullet$ defined for any $a$, $b\in A$ and $x$, $y
\in V$ by:
\begin{equation} \eqlabel{hoproduct2}
(a, \, x) \bullet (b, \, y) := (a\cdot b, \,\, \vartheta (a, \, b)
+ a \triangleright y + b \triangleright x + x \cdot_V y)
\end{equation}
$A \# V$ is called the \emph{crossed product} associated to
$\Theta(A, V)$ if it is a JJ algebra with the multiplication given
by \equref{hoproduct2}. In this case $ \Theta(A, V) =
(\triangleright, \, \vartheta, \, \cdot_V)$ is called a
\emph{crossed system} of $A$ by $V$. Our next result provides the
necessary and sufficient conditions for $A \# V$ to be a crossed
product:

\begin{proposition}\prlabel{hocprod}
Let $A$ be a JJ algebra, $V$ a vector space and $\Theta(A, V) =
(\triangleright, \, \vartheta, \, \cdot_V)$ a crossed data of $A$
by $V$. Then $A \# V$ is a crossed product if and only if the
following compatibility conditions hold for any $a$, $b$, $c \in
A$ and $x$, $y\in V$:
\begin{enumerate}
\item[(J1)] $(V, \, \cdot_V)$ is a JJ algebra and the bilinear map
$\vartheta : A \times A \to V$ is symmetric

\item[(J2)] $(a \cdot b) \triangleright x + a \triangleright (b
\triangleright x) + b \triangleright (a \triangleright x) + x
\cdot_V \vartheta (a, \, b) = 0$

\item[(J3)] $a \triangleright (x \cdot_V \, y) + x \cdot_V (a
\triangleright y) + y \cdot_V (a\triangleright x) = 0$

\item[(J4)] $\sum_{(c)} \, \vartheta (a, \, b\cdot c) + \,
\sum_{(c)} a \triangleright \vartheta (b, \, c) = 0$
\end{enumerate}
\end{proposition}

\begin{proof}
We start with the commutativity condition: one can easily see that
$\bullet$ defined by \equref{hoproduct2} is commutative if and
only if $\vartheta: A \times A \to V$ is symmetric and $\cdot_V :
V \times V \to V$ is commutative. For the rest of the proof it is
worth noticing that since in $A \# V$ we have $(a, x) = (a, 0) +
(0, x)$, the Jacobi condition holds if and only if it holds for
all generators of $A \# V$, i.e. for the set $\{(a, \, 0) ~|~ a
\in A\} \cup \{(0, \, x) ~|~ x \in V\}$. We only provide a sketch
of the proof, the rest of the details being left to the reader.
For instance, the Jacobi condition for the multiplication given by
\equref{hoproduct2} holds in \{(0, x), \, (0, y), \, (a, 0)\} if
and only if (J3) holds. Similarly, the Jacobi condition holds in
\{(a, 0), \, (b, 0),\, (0, z)\} if and only if (J2) holds while
the Jacobi condition holds in  \{(a, 0), \, (b, 0),\, (c, 0)\} if
and only if (J4) holds. Finally,  the Jacobi condition holds in
\{(0, x), \, (0, y), \, (0, z)\} if and only if $\cdot_V : V
\times V \to V$ also satisfies the Jacobi condition. This together
with the commutativity proved below shows that $\cdot_V$ is a JJ
algebra structure on $V$.
\end{proof}

In what follows a crossed system of $A$ by $V$ will be seen as a
system of bilinear maps $\Theta(A, V) = (\triangleright, \,
\vartheta, \, \cdot_V)$ satisfying axioms (J1)-(J4) and the set of
all crossed systems of $A$ by $V$ will be denoted by ${\mathcal
C}{\mathcal S} \, (A, \, V)$.

\begin{examples} \exlabel{abelian}
1. By applying \prref{hocprod} we obtain that a crossed data
$(\triangleright, \, \vartheta, \, \cdot_V)$ for which $\cdot_V$
is the trivial multiplication on $V$ is a crossed system if and
only if $(V, \triangleright)$ is a JJ $A$-module and $\vartheta :
A \times A \to V$ is a symmetric bilinear map satisfying the
compatibility condition (J4). A symmetric bilinear map $\vartheta
: A \times A \to V$ satisfying (J4) will be called a
$\triangleright$-\emph{cocycle} by analogy with Lie algebras
\cite{CE}. Furthermore, if $\vartheta$ is also the trivial map,
then the associated crossed product will be called the
\emph{trivial extension} of the JJ algebra $A$ by the $A$-module
$V$.

2. A crossed system $\Theta(A, V) = (\triangleright, \, \vartheta,
\, \cdot_V)$ for which $\vartheta$ is the trivial map is called a
\emph{semidirect system} of $A$ by $V$. In this case $\vartheta$
will be omitted when writing down $\Theta(A, V)$. Axioms in
\prref{hocprod} boil down to the following: $\Theta(A, V) =
(\triangleright, \, \cdot_V)$ is a semidirect system if and only
if $(V, \, \triangleright)$ is a JJ $A$-module, $(V, \cdot_V)$ is
a JJ algebra and for all $a\in A$, $x$, $y\in V$ we have:
\begin{eqnarray}
a \triangleright (x \cdot_V \, y) + x \cdot_V (a \triangleright y)
+ y \cdot_V (a\triangleright x) = 0 \eqlabel{semy}
\end{eqnarray}
The associated crossed product will be called a \emph{semidirect
product} of JJ algebras and will be denoted by $A \ltimes V := A
\ltimes_{(\triangleright, \, \cdot_V)} V$. The terminology is
motivated in \coref{splialg} by the fact that exactly as in the
case of Lie algebras, this construction describes split
epimorphisms within the category of JJ algebras.
\end{examples}

The crossed product is the tool to answer the GE problem. Indeed,
first we observe that the canonical projection $\pi_A : A \# V \to
A$, $\pi_A (a, x) := a$ is a surjective JJ algebra map with kernel
$\{0\} \times V \cong V$. Hence, the JJ algebra $A \# V$ is an
extension of the JJ algebra $A$ by the JJ algebra $(V, \cdot_V)$
via
\begin{eqnarray} \eqlabel{extenho1}
\xymatrix{ 0 \ar[r] & V \ar[r]^{i_{V}} & A \# \, V
\ar[r]^{\pi_{A}} & A \ar[r] & 0 }
\end{eqnarray}
where $i_V (x) = (0, \, x)$. Conversely, we have:

\begin{proposition}\prlabel{hocechiv}
Let $A$ be a JJ algebra, $E$ a vector space and $\pi : E \to A$ an
epimorphism of vector spaces with $V = {\rm Ker} (\pi)$. Then any
JJ algebra structure $\cdot_E $ which can be defined on the vector
space $E$ such that $\pi : (E, \cdot_E) \to A$ becomes a morphism
of JJ algebras is isomorphic to a crossed product $A \# V$.
Furthermore, the isomorphism of JJ algebras $ (E, \cdot_E) \cong A
\# V$ can be chosen such that it stabilizes $V$ and co-stabilizes
$A$.

Therefore, any JJ algebra structure on $E$ such that $\pi : E \to
A$ is a JJ algebra map is cohomologous to an extension of the form
\equref{extenho1}.
\end{proposition}

\begin{proof}
Let $\cdot_E $ be a JJ algebra structure of $E$ such that $\pi:
(E, \cdot_E) \to A$ is a JJ algebra map. Since we are working over
a field, we can find a $k$-linear section $s : A \to E$ of $\pi$,
i.e. $\pi \circ s = {\rm Id}_{A}$. It follows that $\varphi : A
\times V \to E$, $\varphi (a, x) := s(a) + x$ is an isomorphism of
vector spaces with the inverse $\varphi^{-1} (y) = \bigl(\pi(y),
\, y - s (\pi(y)) \bigl)$, for all $y\in E$. Using the section $s$
we can define the following two bilinear maps:
\begin{eqnarray*}
\triangleright &=& \triangleright_{s} \, \, : A \times V \to V,
\,\,\,\,
a \triangleright x := s(a) \cdot_E x \eqlabel{act2}\\
\vartheta &=& \vartheta_s \,\,\, : A \times A \to V, \,\,\,\,
\vartheta (a, b) := s(a) \cdot_E s(b) - s(a \cdot b)
\eqlabel{coc1}
\end{eqnarray*}
where $a$, $b\in A$ and $x\in V$ and let $\cdot_V : V \times V \to
V$ be the restriction of $\cdot_E $ at $V$, i.e. $x \cdot_V \, y
:= x \cdot_E y$, for all $x$, $y\in V$. It is straightforward to
see that these are well-defined maps. Finally, using the system
$(\triangleright, \, \vartheta, \, \cdot_V)$ connecting $A$ and
$V$ we can prove that the unique JJ algebra structure $\bullet$
that can be defined on the direct product of vector spaces $A
\times V$ such that $\varphi : A \times V \to (E, \cdot_E)$ is an
isomorphism of JJ algebras is given by:
\begin{equation} \eqlabel{hoproduct}
(a, x) \bullet (b, y) := (a\cdot b, \, \vartheta (a, b) + a
\triangleright y + b \triangleright x + x \cdot_V y)
\end{equation}
for all $a$, $b\in A$, $x$, $y \in V$. Indeed, let $\bullet $ be
such a JJ algebra structure on $A\times V$. Then we have:
\begin{eqnarray*}
(a, x) \bullet (b, y) &=& \varphi^{-1} \bigl(\varphi(a, x) \cdot_E
\varphi(b, y)\bigl) = \varphi^{-1} \bigl(  ( s(a) + x) \cdot_E (s(b) + y) \bigl) \\
&=& \varphi^{-1} ( s(a) \cdot_E s(b) + s(a) \cdot_E y + x \cdot_E s(b) + x \cdot_V y ) \\
&=& \bigl (a \cdot b, \,  s(a) \cdot_E s(b) - s (a\cdot b) + s(a) \cdot_E y + x \cdot_E s(b) + x\cdot_V y \bigl) \\
&=& \bigl(a \cdot b, \, \vartheta (a, \, b) + a \triangleright y +
b \triangleright x + x \cdot_V y )
\end{eqnarray*}
which is the desired multiplication. Thus, $\varphi : A \# V \to
(E, \cdot_E)$ is an isomorphism of JJ algebras which stabilizes
$V$ and co-stabilizes $A$. This finishes the proof.
\end{proof}

We mentioned briefly at the end of \exref{abelian} that the
semidirect product of JJ algebras characterizes split epimorphisms
within the category. This is what we prove next:

\begin{corollary} \colabel{splialg}
A JJ algebra map $\pi: B \to A$ is a split epimorphism in the
category of JJ algebras if and only if there exists an isomorphism
of JJ algebras $B \cong A\ltimes V$, where $V = {\rm Ker} (\pi)$
and $A \ltimes V$ is a semidirect product of JJ algebras.
\end{corollary}

\begin{proof}
Firstly, note that in the case of the semidirect product $A\ltimes
V$, the canonical projection $p_A : A \ltimes V \to A$, $p_{A}(a,
\, x) = a$ has a section that is a JJ algebra map defined by $s_A
(a) = (a, \, 0)$, for all $a\in A$. Conversely, let $s: A \to B$
be a JJ algebra map with $\pi \circ s = {\rm Id}_A$. Then, the
bilinear map $\vartheta_s$ constructed in the proof of
\prref{hocechiv} is the trivial map and hence the corresponding
crossed product $A \# V$ reduces to a semidirect product $A
\ltimes V$.
\end{proof}

In light of \prref{hocechiv}, the classification part of the
GE-problem comes down to the classification of the crossed
products associated to all crossed systems of $A$ by $V$. Our next
result parameterizes (iso)morphisms between two such crossed products.

\begin{lemma}\lelabel{HHH}
Let $\Theta(A, V) = (\triangleright, \, \vartheta, \, \cdot_V)$
and $\Theta '(A, V) = (\triangleright ', \, \vartheta ', \,
\cdot_V ')$ be two crossed systems with $A \# V$, respectively $A
\# ' V$, the corresponding crossed products. Then there exists a
bijection between the set of all JJ algebra morphisms $\psi: A \#
V \to A \# ' V$ which stabilize $V$ and co-stabilize $A$ and the
set of all linear maps $r: A \to V$ satisfying the following
compatibilities for all $a$, $b \in A$, $x$, $y \in V$:
\begin{enumerate}
\item[(CH1)] $x \cdot_V y = x \cdot_V ' y$; \item[(CH2)] $a
\triangleright x = a \triangleright ' x + r(a) \cdot_V ' x$;
\item[(CH3)] $\vartheta(a, \,b) + r(a\cdot b) = \vartheta '(a, \,
b) + a \triangleright ' r(b) + b \triangleright ' r(a)  + r(a)
\cdot_V ' r(b)$
\end{enumerate}
Under the above bijection the JJ algebra morphism $\psi =
\psi_{r}: A \# V \to A \# ' V$ corresponding to $r: A \to V$ is
given by $\psi(a, x) = (a, \, r(a) + x)$, for all $a \in A$, $x
\in V$. Moreover, $\psi = \psi_{r}$ is an isomorphism with the
inverse given by $\psi^{-1}_{r} = \psi_{-r}$.
\end{lemma}

\begin{proof} We can easily see that a linear map $\psi: A \times V \to A \times V$
stabilizes $V$ and co-stabilizes $A$ if and only if there exists a
uniquely determined linear map $r: A \to V$ such that $\psi(a, x)
= \psi_{r}(a, \,  x) = (a, \, r(a)+ x)$, for all $a \in A$, $x \in
V$. The proof will be finished once we show that $\psi: A \# V \to
A \# ' V$ is a JJ algebra map if and only if the compatibility
conditions (CH1)-(CH3) hold. We are left to check that the
following compatibility holds for all generators of $A \times V$:
\begin{equation}\eqlabel{algebramap222}
\psi\bigl((a, x) \bullet (b, y)\bigl) = \psi\bigl((a, x)\bigl) \,
\bullet ' \, \psi\bigl((b, y)\bigl)
\end{equation}
A direct computation shows that \equref{algebramap222} holds for
the pair $(a, 0)$, $(b, 0)$ if and only if (CH3) is fulfilled
while \equref{algebramap222} holds for the pair $(0, x)$, $(a, 0)$
if and only if (CH2) is satisfied. Finally, \equref{algebramap222}
holds for the pair $(0, x)$, $(0, y)$ if and only if (CH1) holds.
\end{proof}

It is therefore natural to introduce the following:

\begin{definition}\delabel{echiaa}
Let $A$ be a JJ algebra and $V$ a vector space. Two crossed
systems $\Theta(A, V)=(\triangleright, \, \vartheta, \, \cdot_V)$
and $\Theta'(A, V) = (\triangleright ', \, \vartheta ', \, \cdot_V
')$ are called \emph{cohomologous}, and we denote this by
$\Theta(A, V) \approx \Theta'(A, V)$, if and only if $\cdot_V =
\cdot_V '$ and there exists a linear map $r: A \to V$ such that
for any $a$, $b \in A$, $x$, $y \in V$ we have:
\begin{eqnarray}
a \triangleright x &=& a \triangleright ' x + r(a) \cdot_V ' x \eqlabel{compa2}\\
\vartheta(a,\, b) &=& \vartheta '(a, \, b) + a \triangleright '
r(b) + b \triangleright ' r(a) + r(a) \cdot_V ' r(b) - r(a\cdot b)
\eqlabel{compa3}
\end{eqnarray}
\end{definition}

The theoretical answer to the GE-problem now follows as a
conclusion of this section:

\begin{theorem}\thlabel{main1222}
Let $A$ be a JJ algebra, $E$ a vector space and $\pi : E \to A$ a
linear epimorphism of vector spaces with $V = {\rm Ker} (\pi)$.
Then $\approx$ defined in \deref{echiaa} is an equivalence
relation on the set ${\mathcal C} {\mathcal S} (A, \, V)$ of all
crossed systems of $A$ by $V$. If we denote by ${\mathbb G}
{\mathbb H}^{2} \, (A, \, V) := {\mathcal C}\mathcal{S} (A, V)/
\approx $, then the map
\begin{equation} \eqlabel{thprinc}
{\mathbb G} {\mathbb H}^{2} \, (A, \, V) \to {\rm Gext} \, (A, \,
E), \qquad \overline{(\triangleright, \, \vartheta, \, \cdot_V)}
\, \longmapsto \, A \star_{(\triangleright, \, \vartheta, \,
\cdot_V)} \, V
\end{equation}
is bijective, where $\overline{(\triangleright, \, \vartheta, \,
\cdot_V)}$ denotes the equivalence class of $(\triangleright, \,
\vartheta, \, \cdot_V)$ via $\approx$.
\end{theorem}

\begin{proof} Follows from \prref{hocprod}, \prref{hocechiv} and \leref{HHH}.
\end{proof}

However, to compute the classifying object ${\mathbb G} {\mathbb
H}^{2} \, (A, \, V)$, for a given JJ algebra $A$ and a given
vector space $V$ remains a challenge. We will describe below a
strategy for tackling this computational problem inspired by the
way we defined the equivalence relation $\approx$ in
\deref{echiaa}. Indeed, the aforementioned definition shows that
two different JJ algebra structures $\cdot_V$ and $\cdot_V^{'}$ on
$V$ give rise to two different equivalence classes with respect to
the relation $\approx$ on ${\mathcal C} {\mathcal S} (A, V)$. This
observation will be our starting point. Let us fix $\cdot_{V}$ a
JJ algebra multiplication on $V$ and denote by ${\mathcal C}
{\mathcal S}_{\cdot_{V}} (A, \, V)$ the set of pairs
$\bigl(\triangleright, \, \vartheta \bigl)$ such that $\bigl(
\triangleright, \, \vartheta, \, \cdot_V \bigl) \in {\mathcal C}
{\mathcal S} (A, V)$. Two such pairs $\bigl(\triangleright, \,
\vartheta \bigl)$ and $\bigl(\triangleright', \, \vartheta' \bigl)
\in {\mathcal H} {\mathcal S}_{\cdot_{V}} (A, \, V)$ will be
called $\cdot_V$-cohomologous and will be denoted by
$\bigl(\triangleright, \, \vartheta \bigl) \, \approx_{\cdot_V}
\bigl(\triangleright', \, \vartheta' \bigl)$ if there exists a
linear map $r: A \to V$ such that the compatibility conditions
\equref{compa2}-\equref{compa3} are fulfilled for $\cdot_V' =
\cdot_V$. Then $\approx_{\cdot_V}$ is an equivalence relation on
${\mathcal C} {\mathcal S}_{\cdot_{V}} (A, \, V)$ and we denote by
${\mathbb H}^{2} \, \bigl(A, \, (V, \, \cdot_{V} )\bigl)$ the
quotient set ${\mathcal C} {\mathcal S}_{\cdot_{V}} (A, \, V)/
\approx_{\cdot_V}$. This non-abelian cohomological object
${\mathbb H}^{2} \, \bigl(A, \, (V, \, \cdot_{V} )\bigl)$
classifies all extensions of the given JJ algebra $A$ by the given
JJ algebra $(V, \, \cdot_{V})$ and gives the theoretical answer to
the classical extension problem for JJ algebras in the general
non-abelian case. The above discussion is summarized in the
following result:

\begin{corollary} \colabel{schclas}
Let $(A, \cdot)$ and $(V, \cdot_V)$ be two given JJ algebras.
Then, the map
\begin{equation} \eqlabel{clasextprob}
{\mathbb H}^{2} \, \bigl(A, \, (V, \, \cdot_{V} )\bigl) \to {\rm
Ext} \, (A, \, (V, \, \cdot_V) ), \qquad
\overline{\overline{(\triangleright, \, \vartheta)}} \,
\longmapsto \, A \star_{(\triangleright, \, \vartheta, \,
\cdot_V)} \, V
\end{equation}
is bijective, where ${\rm Ext} \, (A, \, (V, \, \cdot_V) )$ is the
set of equivalence classes of all JJ algebras that are extensions
of the JJ algebra $A$ by $(V, \cdot_V)$ and
$\overline{\overline{(\triangleright, \, \vartheta)}}$ denotes the
equivalence class of $(\triangleright, \, \vartheta)$ via
$\approx_{\cdot_V}$.
\end{corollary}

Furthermore, we obtained the following decomposition of ${\mathbb
G} {\mathbb H}^{2} \, (A, \, V)$:

\begin{corollary} \colabel{desccompcon}
Let $A$ be a JJ algebra, $E$ a vector space and $\pi : E \to A$ an
epimorphism of vector spaces with $V = {\rm Ker} (\pi)$. Then
\begin{equation}\eqlabel{balsoi}
{\mathbb G} {\mathbb H}^{2} \, (A, \, V) = \, \sqcup_{\cdot_{V}}
\, {\mathbb H}^{2} \, \bigl(A, \, (V, \, \cdot_{V} )\bigl)
\end{equation}
where the coproduct on the right hand side is in the category of
sets over all possible JJ algebra structures $\cdot_{V}$ on the
vector space $V$.
\end{corollary}

As it can be easily seen, formula \equref{balsoi} still gives rise
to laborious computations starting with finding all JJ algebra
multiplications on $V$. Of course, the computations become more
complicated when the dimension of $V$ increases. We should point
out that among all components of the coproduct in \equref{balsoi}
the simplest one is that corresponding to the trivial JJ algebra
structure on $V$, i.e. $x \cdot_V y := 0$, for all $x$, $y \in V$
which we shall denote by $V_0 := (V, \, \cdot_V = 0)$. Indeed, let
${\mathcal C} {\mathcal S}_0 (A, \, V_0)$ be the set of all pairs
$\bigl(\triangleright, \, \vartheta \bigl)$ such that $\bigl(
\triangleright, \, \vartheta, \, \cdot_V := 0 \bigl) \in {\mathcal
C} {\mathcal S} (A, V)$. As shown in \exref{abelian}, a pair
$\bigl(\triangleright, \, \vartheta \bigl) \in {\mathcal C}
{\mathcal S}_0 (A, \, V_0)$ if and only if $(V, \,
\triangleright)$ is a JJ $A$-module and $\vartheta : A \times A
\to V$ is a $\triangleright$-cocycle. It turns out by applying
\deref{echiaa} for the trivial multiplication $\cdot := 0$  that
two pairs $\bigl(\triangleright, \, \vartheta \bigl)$ and
$\bigl(\triangleright', \, \vartheta' \bigl) \in {\mathcal C}
{\mathcal S}_0 (A, \, V_0)$ are $0$-cohomologous
$\bigl(\triangleright, \, \vartheta \bigl) \approx_0
\bigl(\triangleright', \, \vartheta' \bigl)$ if and only if
$\triangleright = \triangleright'$ and there exists a linear map
$r: A \to V$ such that for all $a$, $b\in A$ we have:
\begin{equation}\eqlabel{cazslab}
\vartheta(a, \, b) = \vartheta '(a, \, b) + a \triangleright  r(b)
+ b \triangleright r(a) - r(a\cdot b)
\end{equation}
The equality $\triangleright = \triangleright'$ shows that two
different JJ $A$-module structures on $V$ give different
equivalence classes in the classifying object ${\mathbb H}^{2} \,
\bigl(A, \, V_0 \bigl)$. Thus, we can apply the same strategy as
before for computing  ${\mathbb H}^{2} \, \bigl(A, \, V_0 \bigl)$:
we fix $(V, \, \triangleright)$ a JJ $A$-module structure on $V$
and consider the set ${\rm Z}^2_{\triangleright} \, (A, \, V_0) $
of all $\triangleright$-cocycles: i.e. symmetric bilinear maps
$\vartheta : A \times A \to V$ such that
\begin{eqnarray}
\sum_{(c)} \, \vartheta (a, \, b\cdot c) + \sum_{(c)} \, a
\triangleright \, \vartheta (b,\,c) =0 \eqlabel{cocycleb}
\end{eqnarray}
for all $a$, $b$, $c\in A$. Two $\triangleright$-cocycles
$\vartheta$ and $\vartheta'$ will be called cohomologous
$\vartheta \approx_0 \vartheta'$ if and only if there exists a
linear map $r: A \to V$ such that \equref{cazslab} holds. Then
$\approx_0$ is an equivalence relation on ${\rm
Z}^2_{\triangleright} \, (A, \, V_0) $ and the quotient set ${\rm
Z}^2_{\triangleright} \, (A, \, V_0)/ \approx_0$, which we will
denote by ${\rm H}^2_{\triangleright} \, (A, \, V_0)$, plays the
role of the second cohomological group from the theory of
Lie/associative algebras \cite{CE, Hoch2}. We can now put the
above observations together and state the following results which
classifies all extensions of a JJ algebra $A$ by an abelian
algebra $V_0$. It is the JJ algebras counterpart of Hochschild's
result \cite[Theorem 6.2]{Hoch2} for associative algebra and
respectively of Chevalley and Eilenberg's results for Lie algebras
\cite[Theorem 26.2]{CE} concerning the classification of
extensions with an 'abelian' kernel.

\begin{corollary}\colabel{cazuabspargere}
Let $A$ be a JJ algebra and $V$ a vector space viewed with the
trivial JJ algebra structure $V_0$. Then:
\begin{equation}\eqlabel{balsoi2}
{\mathbb H}^{2} \, \bigl(A, \, V_0 \bigl) \, = \,
\sqcup_{\triangleright} \, {\rm H}^2_{\triangleright} \, (A, \,
V_0)
\end{equation}
where the coproduct on the right hand side is in the category of
sets over all possible JJ $A$-module structures $\triangleright$
on the vector space $V$.
\end{corollary}

\section{Applications and Examples} \selabel{aplicatii}
Throughout this section $k$ will be a field of characteristic
$\neq 2, \ 3$. In accordance to our notational conventions we
denote by $k_{0}$ the $1$-dimensional JJ algebra with trivial
multiplication and by $V_0$ the abelian JJ algebra structure on an
arbitrary vector space $V$. In the sequel, the theoretical results
obtained in \seref{crossed} will be applied for two main classes
of JJ algebras: co-flag algebras and metabelian algebras.

\subsection*{Co-flag algebras.}
For a given JJ algebra $A$ we shall classify all JJ algebras $B$
such that there exists a surjective JJ algebra map $\pi : B \to A$
having a $1$-dimensional kernel, which as a vector space will be
assumed to be $k$. We shall compute ${\mathbb G} {\mathbb H}^{2}
\, (A, \, k)$ which classifies all these JJ algebras up to an
isomorphism which stabilizes $k$ and co-stabilizes $A$ as well as
the second classifying object, denoted by ${\mathbb C} {\mathbb P}
\, (A, \, k)$, which provides the classification only up to an
isomorphism of JJ algebras. By computing these two classifying
objects we take the first step towards describing and classifying
a new class of algebras defined as follows:

\begin{definition} \delabel{coflaglbz}
Let $A$ be a JJ algebra and $E$ a vector space. A JJ algebra
structure $\cdot_E$ on $E$ is called a \emph{co-flag JJ algebra
over $A$} if there exists a positive integer $n$ and a finite
chain of surjective morphisms of JJ algebras
\begin{equation} \eqlabel{lant}
A_n : = (E, \cdot_E) \stackrel{\pi_{n}}{\longrightarrow} A_{n-1}
\stackrel{\pi_{n-1}}{\longrightarrow} A_{n-2} \, \cdots \,
\stackrel{\pi_{2}}{\longrightarrow} A_1 \stackrel{\pi_{1}}
{\longrightarrow} A_{0} := A
\end{equation}
such that ${\rm dim}_k ( {\rm Ker} (\pi_{i}) ) = 1$, for all $i =
1, \cdots, n$. A finite dimensional JJ algebra is called a
\emph{co-flag algebra} if it is a co-flag JJ algebra over the JJ
algebra $k_{0}$.
\end{definition}

In order to describe co-flag algebras one more piece of terminology is nedeed:

\begin{definition} \delabel{coflag}
Let $A$ be a JJ algebra. A \emph{co-flag datum of $A$} is a pair
$(\lambda, \vartheta)$ consisting of a linear map $\lambda: A \to
k$ and a symmetric bilinear map $\vartheta : A \times A \to k$
satisfying the following compatibilities for any $a$, $b$, $c\in
A$:
\begin{eqnarray}
&&\lambda(a\cdot b) = -2 \lambda(a) \lambda(b)\eqlabel{ciudat0}\\
&&\sum_{(c)} \, \vartheta (a, \, b\cdot c) + \, \sum_{(c)}
\lambda(a) \vartheta (b, \, c) = 0 \eqlabel{ciudat1}
\end{eqnarray}
\end{definition}

We denote by ${\mathcal C} {\mathcal F} \, (A)$ the set of all
co-flag data of $A$.  As we shall see, ${\mathcal C} {\mathcal F}
\, (A)$ parameterizes the set of all crossed systems of $A$ by a
$1$-dimensional vector space. Our next result provides a
description of the first algebra $A_1$ in the exact sequence
\equref{lant} in terms depending only on $A$.

\begin{proposition}\prlabel{cohocflag}
Let $A$ be a JJ algebra. Then there exists a bijection ${\mathcal
C}{\mathcal S} \, (A, \, k) \cong {\mathcal C} {\mathcal F} \,
(A)$ between the set of all crossed systems of $A$ by $k$ and the
set of all co-flag data of $A$ given such that the crossed product
$A \# k$ associated to $(\lambda, \vartheta) \in {\mathcal C}
{\mathcal F} \, (A)$ is the JJ algebra denoted by $A_{(\lambda,
\vartheta)}$ with the multiplication given for any $a$, $b \in A$,
$x$, $y\in k$ by:
\begin{equation} \eqlabel{patratnul}
(a, x) \bullet (b, y) = \bigl(a\cdot b, \,\, \vartheta (a, b) +
\lambda (a) y + \lambda (b) x \bigl)
\end{equation}
\end{proposition}

\begin{proof}
The proof basically comes down to computing all crossed systems
between $A$ and $k_{0}$, i.e. all bilinear maps $\triangleright :
A \times k \to k$, $ \vartheta : A\times A \to k$ satisfying the
compatibility conditions (J1)-(J4) of \prref{hocprod} with
$\cdot_{k}$ being the trivial multiplication. As $k$ has dimension
$1$ there exists a bijection between the set of all crossed datums
$\bigl( \triangleright, \, \vartheta, \, \cdot_k \bigl)$ of $A$ by
$k$ and the set of pairs $(\lambda, \, \vartheta)$ consisting of a
linear map $\lambda: A \to k$ and a bilinear map $\vartheta: A
\times A \to k$. The bijection is given such that the crossed
datum $( \triangleright, \, \vartheta)$ corresponding to
$(\lambda, \, \vartheta)$ is defined by $a\triangleright x :=
\lambda (a) x$ for all $a \in A$ and $x \in k$. Now it can easily
be seen that (J1) implies $\vartheta$ symmetric while (J3) is
trivially fulfilled. Finally, (J2) amounts to the compatibility
\equref{ciudat0} while \equref{ciudat1} is derived from (J4). The
algebra $A_{(\lambda, \vartheta)}$ is just the crossed product $A
\# k$ associated to this context.
\end{proof}

We will now describe the algebra $A_{(\lambda, \vartheta)}$ in
terms of generators and relations. We will see the elements of $A$
as elements in $A \times k$ through the identification $a = (a, \,
0)$ and denote by $f := (0_A, \, 1) \in A\times k$. Consider $\{
e_i \, | \, i \in I \}$ to be a basis of $A$ as a vector space
over $k$. Then the algebra $A_{(\lambda, \vartheta)}$ has $\{ f,
\, e_i \, | \, i \in I \}$ as a basis for the underlying vector
space while its multiplication $\bullet$ is given for any $i\in I$
by:
\begin{equation}\eqlabel{primaalg}
e_i \bullet e_j = e_i \cdot_A e_j + \vartheta(e_i, \, e_j) \, f,
\,\,\, f^2 = 0, \,\,\, e_i \bullet f = f \bullet e_i =
\lambda(e_i)\,f
\end{equation}
where $\cdot_A$ denotes the multiplication on $A$. Now putting
\prref{cohocflag} and \prref{hocechiv} together yields:

\begin{corollary}\colabel{descefecti}
Let $A$ be a JJ algebra. A JJ algebra $B$ has a surjective JJ
algebra map $B \to A\to 0$ whose kernel is $1$-dimensional if and
only if $B$ is isomorphic to $A_{(\lambda, \vartheta)}$, for some
$(\lambda, \vartheta) \in {\mathcal C} {\mathcal F} \, (A)$.
\end{corollary}

We are now in a position to compute the classifying object ${\mathbb G}
{\mathbb H}^{2} \, (A, \, k)$.

\begin{proposition} \prlabel{hoccal1}
Let $A$ be a JJ algebra. Then,
$$
{\mathbb G} {\mathbb H}^{2} \, (A, \, k) \cong {\mathcal C}
{\mathcal F} \, (A)/ \approx
$$
where $\approx$ is the following equivalence relation on
${\mathcal C} {\mathcal F} \, (A)$: $(\lambda, \vartheta) \approx
(\lambda', \vartheta')$ if and only if $\lambda = \lambda'$ and
there exists a linear map $t: A \to k$ such that for any $a$,
$b\in A$:
\begin{equation} \eqlabel{primaechiv}
\vartheta (a, b) = \vartheta'(a, b) + \lambda(a) t(b) + \lambda(b)
t(a) - t(a\cdot b)
\end{equation}
\end{proposition}

\begin{proof}
We obtain from \thref{main1222} and \prref{cohocflag} that
$$
{\mathbb G} {\mathbb H}^{2} \, (A, \, k) \cong {\mathcal C}
{\mathcal F} \, (A)/ \approx
$$
where the equivalence relation $\approx$ on ${\mathcal C}
{\mathcal F} \, (A)$ is just the equivalence relation $\approx$
from \deref{echiaa} written for the set ${\mathcal C} {\mathcal F}
\, (A)$ via the bijection ${\mathcal H}{\mathcal S} \, (A, \, k)
\cong {\mathcal C} {\mathcal F} \, (A)$ given in
\prref{cohocflag}. The proof is finished by noticing that the equivalence
relation $\approx$ written on the set of all co-flag data takes
precisely the desired form.
\end{proof}

The following decomposition of ${\mathcal C} {\mathcal F} \, (A)/
\approx$ suggested by  \prref{hoccal1}  will come in handy: for a
fixed linear map $\lambda: A \to k$ satisfying \equref{ciudat0} we
shall denote by ${\rm Z}^2_{\lambda} \, (A, \, k)$ the set of all
$\lambda$-cocycles; that is, the set of all symmetric bilinear
maps $\vartheta : A \times A \to k$ satisfying the compatibility
condition \equref{ciudat1}. Two $\lambda$-cocycles $\vartheta$,
$\vartheta' : A \times A \to k$ are equivalent $\vartheta
\approx^{\lambda} \vartheta'$ if and only if there exists a linear
map $t: A \to k$ such that \equref{primaechiv} holds. If we denote
${\rm H}^2_{\lambda} \, (A, \, k) := {\rm Z}^2_{\lambda} \, (A, \,
k)/ \approx^{\lambda}$ we obtain the following decomposition of
${\mathbb G} {\mathbb H}^{2} \, (A, \, k)$:

\begin{corollary} \colabel{hoccal1ab}
Let $A$ be a JJ algebra. Then,
\begin{equation} \eqlabel{astaedefol}
{\mathbb G} {\mathbb H}^{2} \, (A, \, k) \, \cong \,
\sqcup_{\lambda} \,\, {\rm H}^2_{\lambda} \, (A, \, k)
\end{equation}
where the coproduct on the right hand side is in the category of
sets over all possible linear maps $\lambda: A\to k$ satisfying
\equref{ciudat0}.
\end{corollary}

The example below highlights the way \coref{hoccal1ab} works in
order to classify JJ co-flag algebras. We shall denote by ${\rm
Sym} (n, k)$ the vector space of all $n\times n$ symmetric
matrices over $k$ and by ${\rm Sym} (n, k)/kI_n$ the quotient
vector space via the subspace of matrices of the form $a I_n$, for
all $a \in k$.

\begin{example} \exlabel{heiscof}
Let $\mathfrak{h} (2n + 1, k)$ be the $(2n+1)$-dimensional
commutative Heisenberg JJ algebra defined in \exref{frobex}. Then,
$$
{\mathbb G} {\mathbb H}^{2} \, (\mathfrak{h} (2n + 1, k), \, k) \,
\cong \, {\rm Sym} (n, k) \times {\rm Sym} (n, k) \times {\rm Sym}
(n, k)/kI_n
$$
The equivalence classes of all $(2n + 2)$-dimensional JJ algebras
with an algebra projection on $\mathfrak{h} (2n + 1, k)$ are the
following algebras defined for any $(a_{ij}) \times (b_{ij})
\times \overline{(c_{ij})} \in {\rm Sym} (n, k) \times {\rm Sym}
(n, k) \times {\rm Sym} (n, k)/kI_n$ as the vector space having
$\{y, \, e_1, \cdots, e_n, \, f_1, \cdots, f_n, \, z  \}$ as a
basis and the multiplication given for any $i$, $j = \cdots n$ by:
\begin{eqnarray}
e_i \bullet e_j = e_j \bullet e_i := a_{ij} \, y, \quad f_i
\bullet f_j = f_j \bullet f_i := b_{ij} \, y, \quad e_i \bullet
f_j = f_j \bullet e_i := \delta_i^j \, z + c_{ij} \, y
\eqlabel{primex}
\end{eqnarray}
We will prove first that there exists a bijection between
${\mathcal C} {\mathcal F} \, \bigl(\mathfrak{h} (2n + 1, k)
\bigl)$ and the space ${\rm Sym} (n, k)^3$ given such that the
co-flag data $(\lambda, \vartheta) \in {\mathcal C} {\mathcal F}
\, \bigl(\mathfrak{h} (2n + 1, k) \bigl)$ associated to a triple
$(A = (a_{ij}), \, B = (b_{ij}), C =(c_{ij}))$ consisting of
$n\times n$ symmetric matrices is given for any $i$, $j = 1,
\cdots, n$ by:
\begin{eqnarray*}
&& \lambda := 0, \quad \vartheta(e_i, \, e_j) = \vartheta(e_j, \,
e_i) := a_{ij}, \quad \vartheta(f_i, \, f_j) =\vartheta(f_j, \,
f_i) := b_{ij} \\
&& \vartheta(e_i, \, f_j) = \vartheta(e_i, \, f_j) := c_{ij},
\quad \vartheta( -, \, z) = \vartheta( z, \, -) = 0
\end{eqnarray*}
Indeed, let $(\lambda, \vartheta) \in {\mathcal C} {\mathcal F} \,
\bigl(\mathfrak{h} (2n + 1, k) \bigl)$ be a co-flag data. Applying
\equref{ciudat0} for the pair $(e_i, \, e_i)$, $(f_i, \, f_i)$ and
respectively $(e_i, \, f_i)$, taking into account that ${\rm char}
(k) \neq 2$ we obtain $\lambda (e_i) = \lambda (f_i) = \lambda (z)
= 0$, that is $\lambda := 0$, the trivial map. Thus, the
compatibility condition \equref{ciudat1} takes the simplified form
for any $a$, $b$, $c \in \{e_1, \cdots, e_n, f_1, \cdots, f_n, z
\}$:
\begin{equation} \eqlabel{redus}
\sum_{(c)} \, \vartheta (a, \, b \cdot c) = 0
\end{equation}
Applying this equation to the triples $(e_i, \, e_i, \, f_i)$,
$(e_i, \, f_i, \, f_i)$ and respectively $(e_i, \, f_i, \, z)$ we
obtain: $\vartheta (e_i, \, z) = \vartheta (f_i, \, z) = \vartheta
(z, \, z) = 0$, for all $i = 1, \cdots, n$. Thus, any bilinear map
$\vartheta$ of a co-flag datum of $\mathfrak{h} (2n + 1, k)$ is
implemented by three $n\times n$ matrices $(A = (a_{ij}), \, B =
(b_{ij}), C =(c_{ij}))$ via the two-sided formulas:
$$
\vartheta(e_i, \, e_j) :=: a_{ij}, \quad \vartheta(f_i, \, f_j)
:=: b_{ij}, \quad \vartheta(e_i, \, f_j) :=: c_{ij},
$$
for all $i$, $j = 1, \cdots, n$. Of course, $\vartheta$ is
symmetric if and only if $A$, $B$, $C$ are symmetric matrices.
Now, using the multiplication on $\mathfrak{h} (2n + 1, k)$, we
can easily see that such a bilinear map $\vartheta = \vartheta_{A,
B, C}$ satisfies \equref{redus}. Thus, the first part is proven.
We mention that the multiplication $\bullet$ on the JJ algebra
$\mathfrak{h} (2n + 1, k)_{\lambda, \, \vartheta}$ as defined by
\equref{primaalg} takes the form given in \equref{primex}.
Finally, a straightforward computation shows that the equivalence
relation as defined by \equref{primaechiv} of \prref{hoccal1},
takes the following form in the case for triples of symmetric
matrices: two triples of $n\times n$ symmetric matrices $(A, B,
C)$ and $(A', B', C')$ are equivalent $(A, B, C) \approx (A', B',
C')$ if and only if $A = A'$, $B = B'$ and there exists a scalar
$t\in k$ such that $C' - C = t I_n$. This finishes the proof.
\end{example}

The object ${\mathbb G} {\mathbb H}^{2} \, (A, \, k)$ classifies
all crossed products $A \# k$ up to an isomorphism of JJ algebras
which stabilizes $k$ and co-stabilizes $A$. In what follows we
will consider a less restrictive classification: we denote by
${\mathbb C} {\mathbb P} \, (A, \, k)$ the set of JJ algebra
isomorphism classes of all crossed products $A \# k$. Two
cohomologous crossed products $A \# k$ and $A \#' k$ are
isomorphic and therefore there exists a canonical projection
${\mathbb G} {\mathbb H}^{2} \, (A, \, k) \twoheadrightarrow
{\mathbb C} {\mathbb P} \, (A, \, k)$ between the two classifying
objects. Next in line is the task of computing ${\mathbb C}
{\mathbb P} \, (A, \, k)$.

\begin{theorem} \thlabel{clasres}
Let $A$ be a JJ algebra. Then there exists a bijection:
\begin{equation} \eqlabel{hoculmic}
{\mathbb C} {\mathbb P} \, (A, \, k) \cong {\mathcal C} {\mathcal
F} \, (A)/ \equiv
\end{equation}
where $\equiv$ is the equivalence relation on ${\mathcal C}
{\mathcal F} \, (A)$ defined by: $(\lambda, \vartheta) \equiv
(\lambda', \vartheta')$ if and only if there exists a triple
$(s_0, \, \psi, \, r) \in k^* \times {\rm Aut}_{\rm JJ-Alg} (A)
\times {\rm Hom}_k (A, \, k)$ such that
\begin{eqnarray}
\lambda = \lambda' \circ \psi, \quad \vartheta (a, b) \, s_0 =
\vartheta' \bigl(\psi(a), \psi(b) \bigl) + \lambda(a) r(b) +
\lambda(b) r(a) - r(a \cdot b) \eqlabel{2015b}
\end{eqnarray}
for all $a$, $b \in A$.
\end{theorem}

\begin{proof}
From \coref{descefecti} we know that any crossed product $A \# k$
is isomorphic to $A_{(\lambda, \vartheta)}$, for some $(\lambda,
\vartheta) \in {\mathcal C} {\mathcal F} \, (A)$. Thus, we proof
reduces to describing the JJ algebra isomorphisms between two such
algebras $A_{(\lambda, \vartheta)}$. It is a simple fact to notice
that there exists a bijection between the set of all linear maps
$\varphi: A \times k \to A \times k$ and the set of quadruples
$(s_0, \, \beta_0, \, \psi, \, r) \in k\times A \times {\rm Hom}_k
(A, \, A) \times {\rm Hom}_k (A, \, k)$ given such that the linear
map $\varphi = \varphi_{(s_0, \, \beta_0, \, \psi, \, r)}$
associated to $(s_0, \, \beta_0, \, \psi, \, r)$ is defined for
any $a\in A$ and $x\in k$ by:
\begin{equation}\eqlabel{4morf}
\varphi (a, \, x) = \bigl( \psi(a) + x\, \beta_0, \, r(a) + x\,
s_0 \bigl)
\end{equation}

We start by proving that a linear map given by \equref{4morf} is
an isomorphism of JJ algebras from $A_{(\lambda, \vartheta)}$ to
$A_{(\lambda', \vartheta')}$ if and only if $\beta_0 = 0$, $s_0
\neq 0$, $\psi$ is a JJ algebra automorphism of $A$ and
\equref{2015b} holds. It can be easily seen that $\varphi \bigl(
(0, x) \bullet (0, y) \bigl) = \varphi (0, x) \bullet ' \varphi
(0, y)$ if and only if $\beta_0 = 0$, where by $\bullet '$ we
denote the multiplication of $A_{(\lambda', \vartheta')}$. Hence,
in order for $\varphi$ to be a JJ algebra map it should take the
following simplified form for any $a\in A$ and $x\in k$:
\begin{equation}\eqlabel{4morfb}
\varphi (a, \, x) = \bigl( \psi(a), \, r(a) + x\, s_0 \bigl)
\end{equation}
for some triple $(s_0, \, \psi, \, r) \in k \times {\rm Hom}_k (A,
\, A) \times {\rm Hom}_k (A, \, k)$. Next we prove that a linear
map given by \equref{4morfb} is a JJ algebra morphism from
$A_{(\lambda, \vartheta)}$ to $A_{(\lambda', \vartheta')}$ if and
only if $\psi: A \to A$ is a JJ algebra map and the following
compatibilities are fulfilled for any $a$, $b\in A$:
\begin{eqnarray}
&& \lambda (a) \, s_0 = \lambda'\bigl(\psi (a) \bigl) \, s_0,\eqlabel{2015cc} \\
&& r(a \cdot b) + \vartheta (a, b) \, s_0 = \vartheta'
\bigl(\psi(a), \psi(b) \bigl) + \lambda'(\psi(a)) r(b) +
\lambda'(\psi(b) ) r(a) \eqlabel{2015dd}
\end{eqnarray}
Indeed, it is straightforward to see that the compatibility
\equref{2015cc} is exactly the condition $\varphi \bigl( (a, 0)
\bullet (0, x) \bigl) = \varphi (a, 0) \bullet ' \varphi (0, x)$.
Finally, the condition $\varphi \bigl( (a, 0) \bullet (b, 0)
\bigl) = \varphi (a, 0) \bullet ' \varphi (b, 0)$ is equivalent to
the fact that $\psi$ is a JJ algebra endomorphism of $A$ and
\equref{2015dd} holds. We are left to prove that a JJ algebra map
$\varphi = \varphi_{(s_0, \, \psi, \, r)}$ given by
\equref{4morfb} is bijective if and only if $s_0 \neq 0$ and
$\psi$ is a JJ automorphism of $A$. Assume first that $s_0 \neq 0$
and $\psi$ is bijective with the inverse $\psi^{-1}$. Then, we can
see that $\varphi_{(s_0, \, \psi, \, r)}$ is a JJ algebra
isomorphism with the inverse given by $\varphi^{-1}_{(s_0, \,
\psi, \, r)} := \varphi_{(s_0^{-1}, \, \psi^{-1}, \, -(r\circ
\psi^{-1}) s_0^{-1})}$. Conversely, assume that $\varphi$ is
bijective. Then its inverse $\varphi^{-1}$ is a JJ algebra map and
thus has the form $\varphi^{-1} (a, \, x) = (\psi' (a), \, r' (a)
+ x s_0')$, for some triple $(s_0', r', \psi')$. If we write
$\varphi^{-1} \circ \varphi (0, \, 1) = (0, \, 1)$ we obtain that
$s_0 s_0' = 1$ i.e. $s_0$ is invertible in $k$. In the same way
$\varphi^{-1} \circ \varphi (a, \, 0) = (a, \, 0) = \varphi \circ
\varphi^{-1} (a, \, 0)$ gives that $\psi$ is bijective and $\psi'
= \psi^{-1}$.
\end{proof}

\begin{remark} \relabel{n=1}
Let us denote by $\mathfrak{h} (2n + 1, k)_{A, B, C}$ the JJ
algebra defined by \equref{primex}. These are the $(2n +
2)$-dimensional JJ algebras that admit an algebra surjection on
$\mathfrak{h} (2n + 1, k)$. This family of JJ algebras are
classified from the view point of the extension problem, i.e. up
to an isomorphism which stabilizes $k$ and co-stabilizes
$\mathfrak{h} (2n + 1, k)$, in \exref{heiscof} where the
classifying object ${\mathbb G} {\mathbb H}^{2} \, (\mathfrak{h}
(2n + 1, k), \, k)$ is computed. Now, their classification up to a
JJ algebra isomorphism by explicitly computing the second
classifying object, namely ${\mathbb C} {\mathbb P} \,
(\mathfrak{h} (2n + 1, k), \, k)$ as constructed in
\thref{clasres}, seems to be a very difficult task that will be
addressed somewhere else. Even if the compatibility condition
\equref{2015b} takes the simplified form $\vartheta (a, b) \, s_0
= \vartheta' \bigl(\psi(a), \psi(b) \bigl) - r(a \cdot b)$, this
is still more general than the classical Kronecker-Williamson
equivalence of bilinear forms \cite{will, horn}. Thus, the
starting point in addressing this problem should be a new and more
general classification theory for bilinear forms.

For instance, by taking $n = 1$ in \exref{heiscof} we obtain that
${\mathbb G} {\mathbb H}^{2} \, (\mathfrak{h} (3, k), \, k) \,
\cong \, k \times k \times \{0\} \cong k^2$. Thus, any
$4$-dimensional JJ algebra that admits an algebra surjection on
$\mathfrak{h} (3, k)$ is isomorphic to the JJ algebra having $\{y,
\, e, \, f, \, z \}$ as a basis and the multiplication defined by:
$$
e\cdot e := a \, y, \quad f\cdot f := b\, y, \quad e\cdot f := z
$$
for some $a$, $b\in k$. We denote by $(\mathfrak{h} (3, k), \,
k)_{(a, \, b)}$ this JJ algebra. If $k = k^2$, we can easily see
that, up to an isomorphism there are only three such algebras,
namely $(\mathfrak{h} (3, k), \, k)_{(0, \, 0)}$, $(\mathfrak{h}
(3, k), \, k)_{(1, \, 0)}$ and $(\mathfrak{h} (3, k), \, k)_{(1,
\, 1)}$ which appear in \cite[Proposition 3.4]{BF} in different
notational conventions though.
\end{remark}

As a consequence of \thref{clasres} we can derive a classification
result for the semidirect products of JJ algebras of the form $A
\ltimes k$. Recall from \exref{abelian} that a semidirect product
$A \ltimes k$ is just a crossed product $A_{(\lambda, \vartheta)}
= A\# k$ having a trivial $\lambda$-cocycle $\vartheta$. The
algebra obtained in this way will be denoted by $A_{\lambda}$.

\begin{corollary}\colabel{cassemidirect}
Let $A$ be a JJ algebra and $\lambda$, $\lambda' : A \to k$ two
linear maps satisfying \equref{semy}. Then there exists an
isomorphism of JJ algebras $A_{\lambda} \cong A_{\lambda'}$ if and
only if there exists $\psi \in {\rm Aut}_{\rm JJ-Alg} (A)$ such
that $\lambda = \lambda' \circ \psi$.
\end{corollary}

The first step towards determining the automorphism group ${\rm
Aut}_{\rm JJ-Alg} \, (A_{(\lambda, \vartheta)})$, for any
$(\lambda, \vartheta) \in {\mathcal C} {\mathcal F} \, (A)$ was
taken in the proof of \thref{clasres}. Let $k^*$ be the units
group of $k$, $k^* \times {\rm Aut}_{\rm JJ-Alg} (A)$ the direct
product of groups and $(A^*, +)$ the underlying abelian group of
the linear dual $A^* = {\rm Hom}_k (A, \, k)$. Then the map given
for any $s_0 \in k^*$, $\psi \in {\rm Aut}_{\rm JJ-Alg} (A)$ and
$r\in A^*$ by:
$$
\zeta: k^* \times {\rm Aut}_{\rm JJ-Alg} (A) \to {\rm Aut}_{\rm
Gr} \, (A^*, +), \qquad \zeta (s_0, \, \psi) \, (r) := s_0^{-1} \,
r \circ \psi
$$
is a morphism of groups. This enables us to construct the
semidirect product of groups $A^* \, \ltimes_{\zeta} \bigl(k^*
\times {\rm Aut}_{\rm JJ-Alg} (A) \bigl)$ associated to $\zeta$.
The next result describes ${\rm Aut}_{\rm JJ-Alg} (A_{(\lambda,
\vartheta)} )$ as a certain subgroup of the semidirect product
$A^* \, \ltimes_{\zeta} \bigl(k^* \times {\rm Aut}_{\rm JJ-Alg}
(A) \bigl)$.

\begin{corollary} \colabel{izoaut}
Let $A$ be a JJ algebra, $(\lambda, \vartheta) \in {\mathcal C}
{\mathcal F}\, (A)$ a co-flag datum of $A$ and let ${\mathcal G}
\bigl(A, \, (\lambda, \vartheta) \bigl)$ be the set of all triples
$(s_0, \, \psi, \, r) \in k^* \times {\rm Aut}_{\rm JJ-Alg} (A)
\times A^*$ such that for any $a$, $b\in A$:
$$
\lambda = \lambda \circ \psi, \quad \vartheta (a, b) \, s_0 =
\vartheta \bigl(\psi(a), \psi(b) \bigl) + \lambda(a) r(b) +
\lambda(b) r(a) - r(a \cdot b)
$$
Then, there exists an isomorphism of groups ${\rm Aut}_{\rm
JJ-Alg} ( A_{(\lambda, \vartheta)} ) \cong {\mathcal G} \,
\bigl(A, \, (\lambda, \vartheta) \bigl)$, where ${\mathcal G} \,
\bigl(A, \, (\lambda, \vartheta) \bigl)$ is a group with respect
to the following multiplication:
\begin{equation} \eqlabel{graut}
(s_0, \, \psi, \, r) \star (s_0', \, \psi', \, r') := (s_0 s_0',
\, \psi\circ \psi', \, r \circ \psi' + s_0 r' )
\end{equation}
for all $(s_0, \, \psi, \, r)$, $(s_0', \, \psi', \, r') \in \in
{\mathcal G} \, \bigl(A, \, (\lambda, \vartheta) \bigl)$.
Moreover, the canonical map
$$
{\mathcal G} \, \bigl(A, \, (\lambda, \vartheta) \bigl)
\longrightarrow A^* \, \ltimes_{\zeta} \bigl(k^* \times {\rm
Aut}_{\rm JJ-Alg} (A) \bigl), \qquad (s_0, \, \psi, \, r) \mapsto
\bigl (s_0^{-1} r, \, (s_0, \, \psi) \bigl)
$$
in an injective morphism of groups.
\end{corollary}

\begin{proof} It can be easily seen by a routinely computation that ${\mathcal G} \,  \bigl(A, \, (\lambda, \vartheta)
\bigl)$ is a group with respect to the multiplication
\equref{graut}; the unit is $(1, \, {\rm Id}_A, \, 0)$ and the
inverse of an element $(s_0, \, \psi, \, r)$ is $(s_0^{-1}, \,
\psi^{-1}, \, - s_0^{-1} (r\circ \psi^{-1}))$. Furthermore, from
(the proof of) \thref{clasres} we have $\varphi_{(s_0, \psi, r)}
\circ \varphi_{(s_0', \psi', r')} = \varphi_{(s_0 s_0', \, \psi
\circ \psi', \, r\circ \psi' + s_0 r')}$, where $\varphi_{(s_0,
\psi, r)}$ is an automorphism of $A_{(\lambda, \vartheta)}$ given
by \equref{4morfb}. This settles the first statement. Finally, the
last assertion follows by a straightforward computation.
\end{proof}

\subsection*{Metabelian algebras.}
We recall that a JJ algebra $J$ is called \emph{metabelian} if the
derived algebra $J'$ is an abelian subalgebra of $J$, i.e.
$(a\cdot b) \cdot (c \cdot d) = 0$, for all $a$, $b$, $c$, $d \in
J$. If $I$ is an ideal of $J$, then the quotient algebra $J/I$ is
an abelian algebra if and only if $J' \subseteq I$. Thus, $J$ is a
metabelian JJ algebra if and only if it fits into an exact
sequence of JJ algebras
\begin{eqnarray*} \eqlabel{sirmetab}
\xymatrix{ 0 \ar[r] & B \ar[r]^{i} & {J} \ar[r]^{\pi} & A \ar[r] &
0 }
\end{eqnarray*}
where $B$ and $A$ are abelian algebras. Indeed, if $J$ is
metabelian we can take $B := J'$ and $C := J/J'$ with the obvious
morphisms. Conversely, assume that $J$ is an extension of an
abelian algebra $A$ by an abelian algebra $B$. Since $J/i(B) =
J/{\rm Ker}(\pi) \cong A$ is an abelian algebra, we obtain that
$J' \subseteq i(B) \cong B$. Hence, $J'$ is abelian as a
subalgebra in an abelian algebra, i.e. $J$ is metabelian. This
remark together with \prref{hocechiv} proves that a JJ algebra $J$
is metabelian if and only if there exists an isomorphism $J \cong
A \# V$, where $A \# V$ is a crossed product associated to a
crossed system $(\triangleright, \vartheta)$ between an two
abelian algebras $A$ and $V$. Using the results of
\seref{crossed}, this observation allows us to prove a structure
type theorem for metabelian JJ algebras as well as a
classification result for them. This is what we do next. A
\emph{metabelian system} of a vector space $A$ by a vector space
$V$ is a pair $(\triangleright, \vartheta)$ consisting of two
bilinear maps $\triangleright : A \times V \to V$ and $\vartheta :
A \times A \to V$ such that $\vartheta$ is symmetric and
\begin{eqnarray}
a \triangleright (b \triangleright x) = - \, b \triangleright (a
\triangleright x), \qquad \sum_{(c)} a \triangleright \vartheta
(b, \, c) = 0 \eqlabel{metdat}
\end{eqnarray}
for all $a$, $b$, $c\in A$ and $x\in V$. These are exactly the
axioms remaining from (J1)- (J4) of \prref{hocprod} if we consider
both $\cdot_A$ and $\cdot_V$ to be equal to the trivial
multiplication. We denote by ${\mathcal M}{\mathcal A} \, (A, \,
V)$ the set of all metabelian systems of $A$ by $V$. For a pair
$(\triangleright, \vartheta) \in {\mathcal M}{\mathcal A} \, (A,
\, V)$ the associated crossed product $A \# V$ will be denoted by
$A \#_{(\triangleright, \, \vartheta)} \, V$: it has $A \times V$
as the underlying vector space while the multiplication takes the
following simplified form:
\begin{equation} \eqlabel{metprod}
(a, \, x) \cdot (b, \, y) := (0, \,\, \vartheta (a, \, b) + a
\triangleright y + b \triangleright x)
\end{equation}
for all $a$, $b\in A$ and $x$, $y \in V$. Now we fix a bilinear
map $\triangleright : A \times V \to V$ satisfying the first
equation of \equref{metdat} and denote by ${\rm
Z}^2_{\triangleright} \, (A_0, \, V_0) $ the set of all symmetric,
abelian $\triangleright$-cocycles: i.e. ${\rm
Z}^2_{\triangleright} \, (A_0, \, V_0)$ consists of all symmetric
bilinear maps $\vartheta : A \times A \to V$ satisfying the second
equation of \equref{metdat}. Two such $\triangleright$-cocycles
$\vartheta$ and $\vartheta' \in {\rm Z}^2_{\triangleright} \,
(A_0, \, V_0)$ will be called cohomologous $\vartheta \approx_0
\vartheta'$ if and only if there exists a linear map $r: A \to V$
such that for any $a$, $b\in A$:
$$
\vartheta (a, \, b) = \vartheta' (a, \, b) + a\triangleright r(b)
+ b \triangleright r(a)
$$
Then $\approx_0$ is an equivalence relation on ${\rm
Z}^2_{\triangleright} \, (A_0, \, V_0) $ and the quotient set
${\rm Z}^2_{\triangleright} \, (A_0, \, V_0)/ \approx_0$ will be
denoted by ${\rm H}^2_{\triangleright} \, (A_0, \, V_0)$. We also
recall that ${\mathbb H}^{2} \, \bigl(A_0, \, V_0 \bigl)$
classifies up to an isomorphism that stabilizes $V_0$ and
co-stabilizes $A_0$ all metabelian algebras which are extensions
of $A_0$ by $V_0$. We can now put the above observations together
with \coref{cazuabspargere} and we will obtain the following
result on the structure of metabelian algebras as well as their
classification from the view point of the extension problem:

\begin{corollary}\colabel{cazulmeatabclas}
(1) A JJ algebra $J$ is metabelian if and only if there exists an
isomorphism of JJ algebras $J \cong A \#_{(\triangleright, \,
\vartheta)} \, V$, for some vector spaces $A$ and $V$ and a
metabelian system $(\triangleright, \, \vartheta) \in {\mathcal
M}{\mathcal A} \, (A, \, V)$.

(2) If $A$ and $V$ are vector spaces viewed with the abelian
algebra structure $A_0$ and $V_0$, then there exists a bijection:
\begin{equation} \eqlabel{metclasfor}
{\mathbb H}^{2} \, \bigl(A_0, \, V_0 \bigl) \, \cong \,
\sqcup_{\triangleright} \, {\rm H}^2_{\triangleright} \, (A_0, \,
V_0)
\end{equation}
where the coproduct on the right hand side is in the category of
sets over all possible bilinear maps $\triangleright : A \times V
\to V$ such that $a \triangleright (b \triangleright x) = - \, b
\triangleright (a \triangleright x)$, for all $a$, $b\in A$ and
$x\in V$.
\end{corollary}

To get some insight on the above result we explain briefly how it
works for a metabelian JJ algebra of dimension $n$. An invariant
of such algebras is the dimension of the derived algebra. Thus, we
have to fix a positive integer $1 \leq m \leq n$. Any metabelian
algebra of dimension $n$ that has the derived algebra of dimension
$m$ is isomorphic to a crossed product $
k^{n-m}_0\#_{(\triangleright, \, \vartheta)} \, k^m_0$. Thus, the
first computational and laborious part requires computing all
metabelian systems ${\mathcal M}{\mathcal A} \, (k^{n-m}, \, k^m)$
between $k^{n-m}$ and $k^m$. Next in line is the cohomological
object ${\mathbb H}^{2} \, \bigl(k^{n-m}_0, \, k^m_0 \bigl)$ whose
description relies on the decomposition given in
\equref{metclasfor}. Its elements will classify all metabelian
algebras $k^{n-m}_0\#_{(\triangleright, \, \vartheta)} \, k^m_0$
up to an isomorphism which stabilizes $k^m$ and co-stabilizes
$k^{n-m}$. Computing the other classification object, namely
${\mathbb C} {\mathbb P} \, (k^{n-m}, \, k^m)$ which is a quotient
set of ${\mathbb H}^{2} \, \bigl(k^{n-m}_0, \, k^m_0 \bigl)$ seems
to be a very difficult task -- the corresponding problem at the
level of groups is the famous \emph{isomorphism problem} for
metabelian groups, which seems to be connected to Hilbert�s Tenth
problem, i.e. the problem is algorithmically undecidable (cf.
\cite{baums}).

Below we present two generic examples in order to highlight the
difficulty of the problem, namely we will deal with metabelian JJ
algebras having the derived algebra of dimension $1$ (resp.
codimension 1). Henceforth, we denote by ${\rm Sym} (A \times A,
\, k)$ the space of all symmetric bilinear forms on a vector space
$A$.

\begin{proposition} \prlabel{dim1der}
Let $k$ be a field of characteristic $\neq 2 , 3$ and $A$ a vector
space. Then:

(1) There exists a bijection
$$
{\mathbb G} {\mathbb H}^{2} \, \bigl(A_0, \, k \bigl) \, = \,
{\mathbb H}^{2} \, \bigl(A_0, \, k_0 \bigl) \, \cong \, {\rm Sym}
(A \times A, \, k)
$$
given such that the crossed product $A_{\vartheta} := A_0 \# k_0$
associated to $\vartheta \in {\rm Sym} (A \times A, \, k)$ is the
JJ algebra with the multiplication given for any $a$, $b \in A$,
$x$, $y\in k$ by:
\begin{equation} \eqlabel{patratnulab}
(a, x) \bullet (b, y) = \bigl(0, \,\, \vartheta (a, b) \bigl)
\end{equation}
(2) A metabelian JJ algebra has the derived algebra of dimension
$1$ if and only if it is isomorphic to $A_{\vartheta}$, for some
vector space $A$ and $0\neq \vartheta \in {\rm Sym} (A \times A,
\, k)$.

(3) Two JJ algebras $A_{\vartheta}$ and $A_{\vartheta'}$ are
isomorphic if and only if the symmetric bilinear forms $\vartheta$
and $\vartheta'$ are \emph{homothetic}, i.e. there exists a pair
$(s_0, \, \psi) \in k^* \times {\rm Aut}_k (A)$ such that $s_0 \,
\vartheta (a, \, b) = \theta' \bigl( \psi(a), \, \psi (b) \bigl)$,
for all $a$, $b\in A$.
\end{proposition}

\begin{proof} As we already pointed out, since ${\rm char} (k) \neq 3$
the only JJ algebra structure on the vector space $k$ is the
abelian one. Thus ${\mathbb G} {\mathbb H}^{2} \, \bigl(A_0, \, k
\bigl) \, = \, {\mathbb H}^{2} \, \bigl(A_0, \, k_0 \bigl)$. Now
we apply \prref{cohocflag} which gives a bijection ${\mathcal
C}{\mathcal S} \, (A_0, \, k) \cong {\mathcal C} {\mathcal F} \,
(A_0)$. Since ${\rm char} (k) \neq 2$ any map $\lambda$ of a
co-flag datum $(\lambda, \vartheta) \in {\mathcal C}{\mathcal S}
\, (A_0, \, k)$ is trivial, thanks to \equref{ciudat0}. It follows
that the compatibility condition \equref{ciudat1} is automatically
fulfilled for any $\vartheta \in {\rm Sym} (A \times A, \, k)$.
Thus, we have proved that ${\mathcal C} {\mathcal F} \, (A_0)
\cong {\rm Sym} (A \times A, \, k)$. The first part as well as the
statement in $(2)$ follow from here once we observe that the
equivalence relation given by \equref{primaechiv} of
\prref{hoccal1} becomes equality since $A$ is abelian. Finally,
the statement in $(3)$ follows from the classification result of
\thref{clasres} applied for $A_0$ and taking into account that in
any co-flag datum of an abelian JJ algebra we have $\lambda = 0$.
\end{proof}

\begin{remark} \relabel{homotetic}
The classification result established in \prref{dim1der} (3)
reduces the classification problem of metabelian JJ algebras
having the derived algebra of dimension $1$ to the classification
of symmetric forms on a vector space up to the homotethic
equivalence relation. Now, any two isometric symmetric bilinear
forms are homothetic just by taking $s_0 := 1$. Furthermore, if $k
= k^2$ the converse is also true: that is, two symmetric bilinear
forms are homothetic if and only if they are isometric. In this
case, there is a well developed theory of symmetric bilinear forms
\cite{miln}. For instance, if $k = \CC$ the equivalence classes of
the symmetric bilinear forms coincide with the set of all diagonal
matrices with only 1s or 0s on the diagonal (the number of 0s is
the dimension of the radical of the bilinear form). In the case
that $k$ is a field with $k \neq k^2$ a new classification theory
of bilinear forms up to the homotethic relation needs to be
developed.
\end{remark}

\begin{example} \exlabel{complexmet1}
Using \prref{dim1der} and the classification of all complex
symmetric bilinear forms \cite{miln} we obtain the classification
of all $(n+1)$-dimensional metabelian JJ algebras having the
derived algebra of dimension $1$. Up to an isomorphism, there are
precisely $n$ such JJ algebras defined as follows. For any $t = 1,
\cdots, n$ we denote by $J_t$ the JJ algebra having $\{f, e_1,
\cdots, e_n \}$ as a basis and the multiplication given by
$$
e_1 \cdot e_1 = e_2 \cdot e_2 = \cdots = e_t \cdot e_t := f
$$
A complex $(n+1)$-dimensional JJ metabelian algebra $A$ has the
derived algebra of dimension $1$ if and only if $A \cong J_t$, for
some $t = 1, \cdots, n$. Adding to this family the abelian algebra
$k_0^{n+1}$ of dimension $n+1$, we can conclude that ${\mathbb C}
{\mathbb P} \, (k_0^{n}, \, k) = \{k_0^{n+1}, \, J_1, \cdots, J_n
\}$.
\end{example}

Very interesting is the opposite case concerning the description
of all JJ algebras $A$ having an abelian derived algebra $A'$ of
codimension $1$. This is just the Jacobi-Jordan version of the
classification of Lie algebras of dimension $n$ having the Schur
invariant equal to $n - 1$ - see \cite{BC, CT} for details. Let
$V$ be a vector space, $f \in {\rm End}_k (V)$ with $f^2 = 0$ and
$v_0 \in {\rm Ker} (f)$. Let $V_{(f, \, v_0)} := k \times V$ with
the multiplication defined for any $p$, $q \in k$ and $x$, $y \in
V$ by:
\begin{equation} \eqlabel{metcodim1}
(p, \, x) \cdot (q, \, y) := (0, \, pq \, v_0 + p f(y) + q f(x)
\,)
\end{equation}
Then $V_{(f, \, v_0)}$ is a JJ algebra as it can be seen by a
direct computation as well as from the proof of our next result.

\begin{proposition} \prlabel{codim1der}
Let $k$ be a field of characteristic $\neq 2 , 3$. Then:

(1) A JJ algebra has an abelian derived algebra of codimension $1$
if and only if it is isomorphic to a JJ algebra of the form
$V_{(f, \, v_0)}$, for some vector space $V$, $f \in {\rm End}_k
(V)$ with $f^2 = 0$ and $v_0 \in {\rm Ker} (f)$ such that $(f, \,
v_0) \neq (0, \, 0)$.

(2) For any vector space $V$ there exists a bijection:
$$
{\mathbb H}^{2} \, \bigl(k_0, \, V_0 \bigl) \, \cong \, \sqcup_{f}
\,\, {\rm Ker} (f)/{\rm Im} (f)
$$
where the coproduct in the right hand side is in the category of
sets over all possible linear endomorphisms $f \in {\rm End}_k
(V)$ with $f^2 = 0$. The bijection sends any element
$\overline{v_0} \in {\rm Ker} (f)/{\rm Im} (f)$ to the JJ algebra
$V_{(f, \, v_0)}$.
\end{proposition}

\begin{proof}
First of all we note that a JJ algebra has an abelian derived
algebra of codimension $1$ if and only if it is isomorphic to a
crossed product of the form $k_0 \# V_0$, for some vector space
$V$. If $A$ is such an algebra, we can take $V := A'$, the derived
algebra of $A$. Hence, $A/A' \cong k$, since $A'$ has codimension
$1$ in $A$ and thus it is an abelian algebra as ${\rm char} (k)
\neq 3$. Using \prref{hocechiv} we obtain that $A \cong k_0 \#
V_0$. The converse is obvious. Thus, we have to describe the set
of all pairs $(\triangleright, \, \vartheta)$ consisting of
bilinear maps $\triangleright : k \times V \to V$ and $\vartheta :
k \times k \to V$ such that $(\triangleright, \, \vartheta, \,
\cdot_V := 0)$ is a crossed system of $k_0$ by $V$, that is axioms
(J1)-(J4) hold for $\cdot_V := 0$. Of course, bilinear maps
$\triangleright : k \times V \to V$ are in bijection with linear
maps $f \in {\rm End}_k (V)$ via the two-sided formula $1
\triangleright v :=: f(v)$ and bilinear maps $\vartheta : k \times
k \to V$ are in bijection with the set of all elements $v_0$ of
$V$ via the two-sided formula $\vartheta (a, \, b) :=: ab\, v_0$,
for all $a$, $b\in k$. Now, axioms (J1) and (J3) are trivially
fulfilled for $(\triangleright = \triangleright_f, \, \vartheta =
\vartheta_{v_0})$ and since ${\rm char} (k) \neq 2 , 3$ we can
easily see that axioms (J2) and respectively (J4) hold if and only
if $f^2 = 0$ and $f(v_0) = 0$. This proves the first statement:
the condition  $(f, \, v_0) \neq (0, \, 0)$ ensures that the
algebra $V_{(f, \, v_0)}$ is not abelian. The bijection given in
(2) is obtained by applying \coref{cazuabspargere} for the special
case $A := k_0$. We fix $f \in {\rm End}_k (V)$ with $f^2 = 0$ -
which is the same as fixing $\triangleright$ in
\coref{cazuabspargere} via the above correspondence. Then, the
equivalence relation given by \equref{cazslab} applied to the
elements in ${\rm Ker}(f)$ becomes: $v_0 \approx_0 v_0'$ if and
only if there exists an element $\xi \in V$ such that $v_0 - v_0'
= 2 f(\xi)$, that is $v_0 - v_0' \in {\rm Im} (f)$ as the
characteristic of $k$ is $\neq 2$. This proves the second
statement and the proof is now completely finished.
\end{proof}

\begin{example} \exlabel{finalex}
By taking $V: = k^n$ in \prref{codim1der} we obtain the explicit
description of all $(n+1)$-dimensional metabelian JJ algebras
having the derived algebra of dimension $n$. Indeed, we denote by
$ {\mathcal N} (n) := \{ \, X \in {\rm M}_{n}(k) \, \, | \,\, X^2
= 0 \}$ and ${\rm Ker} (X) := \{ \, v_0 \in k^n \, \, | \,\, X v_0
= 0 \}$, for all $X \in {\mathcal N} (n)$. Then:
$$
{\mathbb H}^{2} \, \bigl(k_0, \, k^n_0 \bigl) \, \cong \,
\sqcup_{X \in {\mathcal N} (n)} \,\, {\rm Ker} (f)/\equiv
$$
where $v_0\equiv v_0'$ if and only if there exists $r\in k^n$ such
that $v_0' - v_0 = Xr$. For $X = (x_{ij}) \in {\mathcal N} (n)$
and $\overline{v_0} =\overline {(v_1, \cdots, v_n)}) \in {\rm Ker}
(X)/\equiv$, we denote by $k^{n}_{(X, \, \overline{v_0})}$ the
associated crossed product $k \# k^n$. Then $k^{n}_{(X, \,
\overline{v_0})}$ is the JJ algebra having $\{f, \, e_1, \cdots,
e_n\}$ as a basis and multiplication defined for any $i$, $j = 1,
\cdots, n$ by:
$$
f \cdot f := \sum_{j=1}^n \, v_j \, e_j, \quad f\cdot e_i :=
\sum_{j=1}^n \, x_{ji} \, e_j
$$
Any $(n+1)$-dimensional metabelian JJ algebra having the derived
algebra of codimension $1$ is isomorphic to such an algebra
$k^{n}_{(X, \, \overline{v_0})}$, for some $X \in {\mathcal N}
(n)$ and $\overline{v_0}) \in {\rm Ker} (X)/\equiv$.
\end{example}

\end{document}